\documentclass[12pt,leqno]{amsart}\usepackage[latin9]{inputenc}\usepackage{verbatim,amsthm,amstext,amssymb,color,amscd,bm}\makeatletter\usepackage{hyperref}\hypersetup{colorlinks=true,citecolor=black,linkcolor=black}\usepackage[shortlabels]{enumitem}\numberwithin{equation}{section}\theoremstyle{plain}\newtheorem{theorem}{Theorem}[section]\newtheorem{lemma}[theorem]{Lemma}\theoremstyle{definition}\makeatother
\begin{document}

\title[Topological entropy, mean dimension, and weakly equivalent flows]{Topological entropy, mean dimension,\\and weakly equivalent flows}
\author[Lei Jin]{Lei Jin}
\address{Lei Jin: School of Mathematics, Sun Yat-sen University, Guangzhou, China}
\email{jinleim@mail.ustc.edu.cn}
\author[Yixiao Qiao]{Yixiao Qiao}
\address{Yixiao Qiao (Corresponding author): School of Mathematics and Statistics, HNP-LAMA, Central South University, Changsha, China}
\email{yxqiao@mail.ustc.edu.cn}
\thanks{Y. Qiao was supported by NNSF of China (Grant No. 12371190).}
\begin{abstract}
In this paper, we mainly revisit a nice theory for topological entropy of weakly equivalent flows, which was originally investigated by Ohno in 1980. We will develop a new approach, being more straightforward and elementary than the measure-theoretic one provided by Ohno, to the theory for weak equivalence of flows, and as a novelty, we study both topological entropy and mean dimension with a highly unified process in relation to such objects.

In particular, for weakly equivalent flows without fixed points we recover Ohno's theorem for topological entropy relation with a substantially different method, and moreover, carry out an analogue within the framework of mean dimension; while for weakly equivalent flows with fixed points, our technique refines the procedure suggested in Ohno's construction, and strengthens Ohno's example with a view towards topological complexity of dynamical systems.

Essentially, our method is topological. For this purpose, we first introduce a modification to the definition of topological entropy, by relating a spanning set inside the state space to a finite set outside, which comes to be a tiny but key difference, via a map, and further, show that it leads eventually to the same value as the topological entropy. Although to an intermediate extent, the modified quantity generally does not coincide with the one appearing in the classical definition, it has some basic combinatorial feature, which not only applies flexibly to topological entropy, but also adapts directly to the mean dimension context. Using this alternative, we are allowed to re-establish and enrich Ohno's theory.
\end{abstract}
\maketitle

\medskip

\section*{Contents}
\begin{itemize}
\item[1.]
Introduction
\begin{itemize}
\item[1.1.]
Background and motivation
\item[1.2.]
Overview of content and structure
\end{itemize}
\item[2.]
An elementary preparation for time-reparameterization
\item[3.]
An alternative approach to topological entropy
\begin{itemize}
\item[3.1.]
Bowen's definition
\item[3.2.]
A modified approach to topological entropy
\item[3.3.]
How to use this definition?
\item[3.4.]
Definition of topological entropy for flows
\item[3.5.]
Review of mean dimension
\end{itemize}
\item[4.]
Weakly equivalent flows without fixed points: mean dimension relation
\item[5.]
Weakly equivalent flows without fixed points: topological entropy relation
\begin{itemize}
\item[5.1.]
An elementary counting lemma
\item[5.2.]
Estimate for $m$
\item[5.3.]
Estimate for $M$
\end{itemize}
\item[6.]
Weakly equivalent flows with fixed points: a stronger example
\end{itemize}

\medskip

\bigskip

\section{Introduction}
\subsection{Background and motivation}
In this paper, we mainly revisit a nice theory for topological entropy of weakly equivalent flows, which was originally explored by Ohno \cite{O80}.

A \textbf{flow} $\Phi$ defined on a compact metrizable space $X$ is a continuous map $\Phi:\mathbb{R}\times X\to X$, $(t,x)\mapsto\phi_t(x)$, satisfying that $\phi_0(x)=x$ and $\phi_{t+s}(x)=\phi_t(\phi_s(x))$, for all $t,s\in\mathbb{R}$ and $x\in X$. Sometimes we write $(\phi_t)_{t\in\mathbb{R}}$ for the flow instead of $\Phi$, which is also referred to as a one-parameter group of homeomorphisms of $X$, being continuous in $(t,x)$. Such a flow is usually denoted by a pair $(X,\Phi)$ or $(X,(\phi_t)_{t\in\mathbb{R}})$. Two flows $(X,\Phi)$ and $(Y,\Psi)$ are \textbf{weakly equivalent} if there is a homeomorphism from $X$ to $Y$ mapping orbits of $\Phi$ onto orbits of $\Psi$ (while preserving the time-orientation for each orbit). Clearly, the notion of weak equivalence can be understood as a relation weaker than topological conjugacy (i.e. being dynamically isomorphic).

In 1980, Ohno \cite{O80} systematically investigated the connection between two weakly equivalent flows with a view towards topological entropy. Using measure-theoretic tools with the help of the variational principle and ergodic theorems, Ohno \cite{O80} showed that the three properties of topological entropy equal to zero, infinite, and being positive and finite, respectively, are all preserved by weak equivalence for flows without fixed points, while to the contrary, neither of them is the case for flows with fixed points (by constructing an example of a pair of weakly equivalent flows of topological entropy zero and positive, respectively).

A remaining problem being quite unclear at that moment was posed naturally by Ohno \cite{O80}, which roughly asked if one can take a better structure on the flows into account so as to preserve zero topological entropy as an invariant under such a weak equivalence. In 2009, Sun--Young--Zhou \cite{SYZ09} found a solution to this problem. More precisely, they constructed two weakly equivalent smooth flows, one of which has positive topological entropy and the other has zero topological entropy. This nice construction clarifies a fact that such phenomena still exist firm even within the class of differentiable flows, i.e., being not affected by some very good structure, in some sense, equipped on the flow.

The above-mentioned direction is to push ahead with nice structures equipped on a pair of weakly equivalent flows. In spite of that, it focused mainly on the differentiability of the dynamical mapping along with the smoothness of its state space (e.g., a compact Riemannian manifold), but made a minor contribution to the topological aspect. More specifically, it does not develop any topological complexity theory for weak equivalence (i.e., limiting one's attention only to topological entropy of flows). In this paper, we aim to explore weakly equivalent flows along another direction of dynamics: the topological complexity. In fact, the topological complexity of two weakly equivalent flows may be extremely far away from each other, and in particular, the possible difference between two flows with fixed points in the complexity extent to which weak equivalence may preserve, turns out to become much more involved than what ``being of topological entropy zero or infinity'' can describe, which in consequence has to be demonstrated beyond the framework of entropy theory.

Now a reasonable direction is to examine some aspect of topological complexity other than topological entropy, for weakly equivalent flows. Among topological invariants of dynamical systems, mean dimension, which originates with Gromov--Lindenstrauss--Weiss around 2000, has proved to achieve significant success, as remarkable as topological entropy. For example, the topological entropy of the shift action over the alphabet $[0,1]$ (being also referred to as the full shift on the product space $[0,1]^\mathbb{Z}$) is equal to infinity, the same as that over a countably infinite alphabet. However, these two dynamical systems seem to be essentially distinct, as the former is obviously much more complicated than the latter. From this point of view, topological entropy is somewhat insufficient to reflect the complexity difference between them. It turns out that mean dimension provides a proper description greatly. In relation to the above-mentioned dynamical systems, for instance, the mean dimension of the former is $1$, while the latter $0$ (for details the reader is referred to \cite{LW00}). The example is very concrete and simple, and in general, it was shown \cite{LW00} that if a dynamical system has finite topological entropy, then its mean dimension must be zero, but not vice versa. Seen from this aspect, mean dimension substantially benefits a desired characterization for dynamical systems having infinite topological entropy. Moreover, those invariants for characterizing topological complexities can be also regarded as a kind of dynamical analogues of cardinality and dimension.

As pointed out previously, the topological complexity of weakly equivalent flows with fixed points may differ strikingly, and indeed, we show that there exists a pair of weakly equivalent flows such that one of them has infinite mean dimension while the other has zero topological entropy (for details see Section 6). However, for a pair of weakly equivalent flows without fixed points, when they are at the level of having topological entropy infinity, we should naturally turn to considering their mean dimension, more precisely, we prove that in this case mean dimension is still well behaved as topological entropy under the weak equivalence (see Section 4).

In addition, there are at least two more reasons behind our motivation. The first is that all the statements of Ohno's results, as mentioned above, are purely topological, but Ohno's arguments actually adopted a measure-theoretic method, in particular, both the variational principle and the Birkhoff ergodic theorem were used essentially among those proofs, hence it would be satisfactory if one could (re)produce direct proofs (being substantially different from Ohno's proofs, and moreover, relying on neither the variational principle nor the ergodic theorems) of these results, which are purely topological, too. The second is that if one tries to find an analogue of Ohno's theory for mean dimension, then one probably has to develop a new approach to this aim, which meanwhile is hoped to adapt easily to the context of mean dimension.

Motivated by all what we have explained, the main purpose of this paper, being further than our initial goal, is to develop a more flexible approach to topological entropy (see Section 3) in company with a convenient application (being further developed in Section 5) to reestablishing and enriching Ohno's theory for weak equivalence of flows (in whole, not only to the existence of that kind of concrete examples). As expected, this method for estimating topological entropy proves to be more straightforward than the measure-theoretic one provided by Ohno in \cite{O80}, and as a novelty, both topological entropy and mean dimension are related to weakly equivalent flows for which the theories are now studied and carried out with a highly unified process. Lastly, it would be worth mentioning that our technique is new even if one's attention is completely restricted to topological entropy rather than mean dimension (i.e., for anyone not interested in mean dimension it still can be read properly with mean dimension being simply removed from the present paper but without loss of novelty), and that this paper is self-contained (in either sense of them).

\subsection{Overview of content and structure}
The remaining part of this paper is organized as follows:

In Section 2, we prepare material on time-reparameterization for weakly equivalent flows without fixed points, which will be used in Section 4 and Section 5. Throughout this paper, neither a measure nor any measure-theoretic tool is needed, and as a result, we shall limit our attention only to the topological aspect and restrict our preparation to a rather elementary level.

In Section 3, before we proceed any further, we first reformulate the definition of topological entropy (for homeomorphisms) by following the standard procedure explored by Bowen, and then we turn to introducing a (slightly) different approach to topological entropy by making modifications to Bowen's classical definition. Next, we prove that the modified definition for topological entropy leads eventually to the same value as the one defined by the original mechanism, and furthermore, we show the reader how to use this definition (in particular, we reproduce proofs of two well-known propositions). Along this way we proceed to the definition of topological entropy for flows (the value given by such an alternative definition, as expected, certainly agrees with the classically-defined topological entropy, too), and moreover, we explain some extra advantage of our method in defining topological entropy for flows (and in finding possible applications). Finally, a subsection for a short review of mean dimension is included also in this section.

In Section 4, we investigate the mean dimension relation between two weakly equivalent flows without fixed points. The first purpose of this section is of course to carry out an analogue of Ohno's theory for mean dimension. However, a more important purpose is to roughly demonstrate how to put Ohno's theory into our framework, which also can be regarded as a conceptual picture of Section 5, so as to generally describe a purely topological proof of the result.

In Section 5, we study the topological entropy relation for weakly equivalent flows without fixed points. Here the case proves to be more complicated than the one in Section 4. Using the modified approach to the definition of topological entropy for flows introduced in Section 3, we develop an elementary and direct method, and in consequence we are then allowed to re-establish Ohno's theory for topological entropy with it, where the reader may find exactly how this tiny but key difference gives rise to an extra convenience and a more flexible handling of estimates for topological entropy in such a situation.

In Section 6, topological entropy and mean dimension come together, being involved in a stronger example constructed for the purpose of clarifying that in relation to weakly equivalent flows with fixed points, the dynamical complexity difference between them is possible to be extremely far from each other. The proof given (for those desired properties of the construction) in the section is also purely topological and straightforward. Notice that this section is logically independent of Section 2, Section 4 and Section 5.

\section{An elementary preparation for time-reparameterization}
Let $(X,(\phi_t)_{t\in\mathbb{R}})$ and $(Y,(\psi_s)_{s\in\mathbb{R}})$ be two weakly equivalent flows without fixed points via a homeomorphism $\pi:X\to Y$. In this section we shall explain how to find a time change (i.e. time-reparameterization) induced naturally by $\pi:X\to Y$. Specifically, for any $t\in\mathbb{R}$ and $x\in X$ there exists some $\theta(t,x)\in\mathbb{R}$ such that $\psi_{\theta(t,x)}(\pi(x))=\pi(\phi_t(x))$. More precisely, $\theta(t,x)$ is uniquely determined when the point $x\in X$ is free (i.e., $x$ is not a periodic point) with respect to $(\phi_t)_{t\in\mathbb{R}}$. For a periodic point $x\in X$ (but not a fixed point) under $(\phi_t)_{t\in\mathbb{R}}$ there is a fundamental period $r(x)>0$ for $x$, namely the smallest positive number among all the real numbers $r$ with the property that $\phi_r(x)=x$. The same is still valid for $y=\pi(x)\in Y$ under $(\psi_s)_{s\in\mathbb{R}}$, whose fundamental period is denoted by $r(y)>0$. Now we correspond $\theta(n\cdot r(x),x)$ to $n\cdot r(y)$ for every $n\in\mathbb{Z}$, and next, on each interval $n\cdot r(x)\le t<(n+1)\cdot r(x)$ we apply the same argument to $\theta(t,x)$ as derived in the free case.

Note that, however, if a flow $(X,(\phi_t)_{t\in\mathbb{R}})$ contains a fixed point $x\in X$, then the set $\{r\in\mathbb{R}:\phi_r(x)=x\}$ does not form a discrete subgroup of $\mathbb{R}$ (which is equal to $\mathbb{R}$ itself), namely, such a minimum positive period of $x$ does not exist. This comes to be the reason why we assumed the absence of the fixed points of the flows.

Clearly, this defines a map $\theta:\mathbb{R}\times X\to\mathbb{R}$ satisfying some standard properties listed below:
\begin{itemize}
\item[(i)]
$\psi_{\theta(t,x)}(\pi(x))=\pi(\phi_t(x))$, for every $t\in\mathbb{R}$ and $x\in X$.
\item[(ii)]
$\theta(0,x)=0$ and $\theta(t,x)$ is strictly increasing in $t\in\mathbb{R}$, for any $x\in X$.
\item[(iii)]
$\theta(t^\prime+t,x)=\theta(t^\prime,\phi_t(x))+\theta(t,x)$, for all $t,t^\prime\in\mathbb{R}$ and $x\in X$.
\end{itemize}
As follows we show that this map satisfies an additional condition:
\begin{itemize}
\item[(iv)]
$\theta$ is continuous in $(t,x)\in\mathbb{R}\times X$.
\end{itemize}
Note that all these properties are being used later.

In order to prove that $\theta:\mathbb{R}\times X\to\mathbb{R}$ is continuous. Fix two sequences $\{t_n\}_{n=1}^\infty$ converging to $t$ in $\mathbb{R}$ and $\{x_n\}_{n=1}^\infty$ converging to $x$ in $X$, respectively. By properties (ii) and (iii) as stated above, one may assume, without loss of generality, that $t=0$, and it suffices to show that $\theta(t_n,x_n)$ tends to $0$ as $n$ goes to $\infty$.

For each $n\in\mathbb{N}$ take any point $x_n^\prime$ between $0$ and $t_n$ from the orbit of $x_n$ in $X$ under $(\phi_t)_{t\in\mathbb{R}}$, i.e. being of the form $x_n^\prime=\phi_{t_n^\prime}(x_n)$ for $t_n^\prime$ ranging over $[0,t_n]$ or $[t_n,0]$. Since $\phi_t(x)$ is continuous in $(t,x)\in\mathbb{R}\times X$ and since $\phi_0(x)=x$, one gets that $x_n^\prime\to x$ in $X$ as $n\to\infty$.

Note that $\pi:X\to Y$ is a homeomorphism. Put $y=\pi(x)$ and $y_n=\pi(x_n)$ for every $n\in\mathbb{N}$. Since $x_n\to x$ in $X$ as $n\to\infty$, $y_n\to y$ in $Y$ as $n\to\infty$. Note that the homeomorphism $\pi:X\to Y$ preserves orbits induced by $(\phi_t)_{t\in\mathbb{R}}$ and $(\psi_s)_{s\in\mathbb{R}}$ along time-orientation. For each $n\in\mathbb{N}$ put $s_n=\theta(t_n,x_n)$. For each $n\in\mathbb{N}$ take any point $y_n^\prime$ between $0$ and $s_n$ from the orbit of $y_n$ in $Y$ under $(\psi_s)_{s\in\mathbb{R}}$, i.e. being of the form $y_n^\prime=\psi_{s_n^\prime}(y_n)$ for $s_n^\prime$ ranging over $[0,s_n]$ or $[s_n,0]$. By properties (i) and (ii) one deduces that $y_n^\prime\to y$ in $Y$ as $n\to\infty$.

Now argue by contradiction. Suppose that $s_n$ does not converge to $0$ as $n\to\infty$. Then there is some $\epsilon_0>0$ with a subsequence of $\{s_n\}_{n=1}^\infty$, which is contained in $\mathbb{R}\setminus[-\epsilon_0,\epsilon_0]$, being still assumed to be $\{s_n\}_{n=1}^\infty$, and further, without loss of generality one may assume that $s_n\ge\epsilon_0$ for all $n\in\mathbb{N}$. It follows that for any $0<\epsilon<\epsilon_0$, $\psi_\epsilon(y_n)\to y$ in $Y$ as $n\to\infty$, and hence, one has $\psi_\epsilon(y)=y$. This implies that $y\in Y$ is a $(\psi_s)_{s\in\mathbb{R}}$-fixed point, a contradiction. Thus, we conclude.

We fix the notation as mentioned above. Now put $m=\inf_{x\in X}\theta(1,x)$ and $M=\sup_{x\in X}\theta(1,x)$. Since $X$ is compact, we see that $0<m\le M<+\infty$, and actually, $m=\min_{x\in X}\theta(1,x)$ and $M=\max_{x\in X}\theta(1,x)$. Next we prove an elementary lemma which is also being used in our main proof. Unlike Ohno's traditional measure-theoretical method, note that here the upper and lower bounds for $\theta(n,x)/n$ are valid for all the points $x\in X$ and every $n\in\mathbb{N}$.

\begin{lemma}\label{mM}
$$m\le\frac{\theta(n,x)}{n}\le M,\quad\quad\forall\;n\in\mathbb{N},\quad\forall\;x\in X.$$
\end{lemma}
\begin{proof}
By property (iii) $$\theta(n,x)=\theta(1,\phi_{n-1}(x))+\theta(n-1,x),\quad\quad\forall n\in\mathbb{Z},\quad\forall x\in X,$$ which implies that $$m\le\theta(n,x)-\theta(n-1,x)\le M,\quad\quad\forall n\in\mathbb{Z},\quad\forall x\in X.$$ Note that for any $x\in X$, $\theta(0,x)=0$. The desired statement follows.
\end{proof}

\section{An alternative approach to topological entropy}
In this section, we first review Bowen's definition of topological entropy for homeomorphisms, and next, introduce a modification to it along with an elementary demonstration of a very simple application. In such a modified approach to topological entropy for homeomorphisms, we are then allowed to define topological entropy for flows with a highly unified procedure. Lastly, we include the definition of mean dimension for homeomorphisms and flows.

In relation to Bowen's famous definition, we borrow the notation and description for the definition of topological entropy generally from a good book by Walters \cite{W82}. However, in order to formulate the notions in a unified way being more proper to our context, we need adopt a slightly different terminology. The reason behind this is that on the one hand, we are dealing with both discrete and continuous time parameters (i.e., both $\mathbb{Z}$-actions and $\mathbb{R}$-actions), and on the other hand, the treatment is expected, for convenience, to be parallel to our modification to topological entropy along with the framework of mean dimension.

The symbol $\mathbb{N}$ is to denote the set of positive integers. The symbol $\#A$ is to denote the cardinality of a set $A$.

\subsection{Bowen's definition}
Let $T:X\to X$ be a homeomorphism defined on a compact metrizable space $X$. Take a compatible metric $d$ on $X$.

For every $\epsilon>0$ a subset $A$ of $X$ is called an \textbf{$\epsilon$-spanning set for $X$ with respect to the distance $d$} if for any $x\in X$ there is some $a\in A$ such that $d(a,x)<\epsilon$. We define $\mathrm{span}_\epsilon(X,d)$ to be the smallest cardinality $\#A$ of any $\epsilon$-spanning set $A\subset X$ for $X$ with respect to $d$. Note that this quantity does not involve any dynamics. By the compactness of $X$, $\mathrm{span}_\epsilon(X,d)$ is obviously a finite and positive integer.

For a finite subset $F$ of $\mathbb{Z}$ the \textbf{Bowen metric generated by $F$ under $T$}, denoted by $d^T_F$ on $X$, is defined as $$d^T_F(x,x^\prime)=\max_{i\in F}d(T^ix,T^ix^\prime),\quad\quad x,x^\prime\in X.$$ Since $F$ is finite, the above expression for $d^T_F$ indeed attains its maximum, and moreover, the distance $d^T_F$ still gives the topology of $X$. Here we do not refer any set $A\subset X$ to an $(n,\epsilon)$-spanning subset of $X$ (for $n\in\mathbb{N}$) so as to avoid some possible confusion.

The \textbf{topological entropy} of $(X,T)$ is defined by $$\mathrm{h}_\mathsf{top}(X,T)=\lim_{\epsilon\to0}\limsup_{\mathbb{N}\ni n\to\infty}\frac{\log\mathrm{span}_\epsilon(X,d^T_{[0,n-1]\cap\mathbb{Z}})}{n}.$$ It is well known that the defined value $\mathrm{h}_\mathsf{top}(X,T)$ does not depend upon the choice of metrics $d$ compatible with the topology equipped on $X$.

\subsection{A modified approach to topological entropy}
Now we make a slight modification to $\mathrm{span}_\epsilon(X,d)$. This will be used essentially in Section 5. As usual, let $T:X\to X$ be a homeomorphism defined on a compact metrizable space $X$, and take a compatible metric $d$ on $X$.

For every $\epsilon>0$ we denote by $\mathrm{part}_\epsilon(X,d)$ the smallest cardinality of any set $B$ such that there exists a map $f:X\to B$ satisfying that $f(x)=f(x^\prime)$ implies $d(x,x^\prime)\le\epsilon$, for any $x,x^\prime\in X$. It is worth mentioning that $f:X\to B$ is \textit{only} a map between two \textit{sets}, which is \textit{not} equipped with any structure (neither topological nor dynamical). Since $X$ is compact, $\mathrm{part}_\epsilon(X,d)$ is a finite positive integer. Note that at the moment this quantity does not involve dynamics. Moreover, instead of an $\epsilon$-spanning set, we prefer to call such a map $f:X\to B$ (along with the set $B$) an \textbf{$\epsilon$-partition for $X$ with respect to the distance $d$}.

Unlike spanning sets, an interesting point here is that, in some sense, the representatives are moving outside via a map. The lemmas below will show that by replacing $\mathrm{span}_\epsilon(X,d^T_{[0,n-1]\cap\mathbb{Z}})$ with $\mathrm{part}_\epsilon(X,d^T_{[0,n-1]\cap\mathbb{Z}})$ the re-defined value coincides with the originally-defined one, and further, as an additional bonus, the intermediate limit (in $n\in\mathbb{N}$) still exists. As we will see in the sequel, the map (as an auxiliary) along with the existence of the intermediate limit proves to be a tiny but key difference that gives rise to a more flexible and convenient handling of estimates for topological entropy in some situation.

\begin{lemma}\label{sts}
For any compact metric space $(X,d)$ and any $\epsilon>0$$$\mathrm{span}_\epsilon(X,d)\le\mathrm{part}_\epsilon(X,d)\le\mathrm{span}_{\epsilon/2}(X,d).$$
\end{lemma}
\begin{proof}
This follows directly from definition.
\end{proof}

Note that the statement of this lemma does not involve dynamics. When involving dynamics one has in particular that for all $\epsilon>0$ and $n\in\mathbb{N}$, $\mathrm{span}_\epsilon(X,d^T_{[0,n]\cap\mathbb{Z}})\le\mathrm{part}_\epsilon(X,d^T_{[0,n]\cap\mathbb{Z}})\le\mathrm{span}_{\epsilon/2}(X,d^T_{[0,n]\cap\mathbb{Z}})$.

\begin{lemma}\label{tnm}
For any $n,m\in\mathbb{N}$ and $\epsilon>0$, $$\mathrm{part}_\epsilon(X,d^T_{[0,n+m-1]\cap\mathbb{Z}})\le\mathrm{part}_\epsilon(X,d^T_{[0,n-1]\cap\mathbb{Z}})\cdot\mathrm{part}_\epsilon(X,d^T_{[0,m-1]\cap\mathbb{Z}}).$$
\end{lemma}
\begin{proof}
Fix $\epsilon>0$. For every $n\in\mathbb{N}$ take a map $f_n:X\to B_n$ such that $\#B_n=\mathrm{part}_\epsilon(X,d^T_{[0,n-1]\cap\mathbb{Z}})$ and $f_n$ is an $\epsilon$-partition for $X$ with respect to the distance $d^T_{[0,n-1]\cap\mathbb{Z}}$. For $n,m\in\mathbb{N}$ consider the map $g_{n+m}:X\to B_n\times B_m$ defined by $x\mapsto(f_n(x),f_m(T^nx))$. Clearly, if $x,x^\prime\in X$ satisfy $g_{n+m}(x)=g_{n+m}(x^\prime)$, then $d^T_{[0,n+m-1]\cap\mathbb{Z}}(x,x^\prime)\le\epsilon$. Thus, $\mathrm{part}_\epsilon(X,d^T_{[0,n+m-1]\cap\mathbb{Z}})\le\#(B_n\times B_m)$, and therefore the statement follows.
\end{proof}

By Lemma \ref{tnm}, the term $\log\mathrm{part}_\epsilon(X,d^T_{[0,n-1]\cap\mathbb{Z}})$ is subadditive in $n\in\mathbb{N}$: $$\log\mathrm{part}_\epsilon(X,d^T_{[0,n+m-1]\cap\mathbb{Z}})\le\log\mathrm{part}_\epsilon(X,d^T_{[0,n-1]\cap\mathbb{Z}})+\log\mathrm{part}_\epsilon(X,d^T_{[0,m-1]\cap\mathbb{Z}}).$$ Hence, by Fekete's lemma, the limit $\lim_{\mathbb{N}\ni n\to\infty}(\log\mathrm{part}_\epsilon(X,d^T_{[0,n-1]\cap\mathbb{Z}}))/n$ exists. Note that the term $\log\mathrm{part}_\epsilon(X,d^T_{[0,n-1]\cap\mathbb{Z}})$ is monotone in $\epsilon>0$. By Lemma \ref{sts}, the topological entropy of $(X,T)$ can be recovered by $$\mathrm{h}_\mathsf{top}(X,T)=\lim_{\epsilon\to0}\lim_{\mathbb{N}\ni n\to\infty}\frac{\log\mathrm{part}_\epsilon(X,d^T_{[0,n-1]\cap\mathbb{Z}})}{n}.$$ Furthermore, although a distance $d$ on $X$ is involved in the term $\mathrm{part}_\epsilon(X,d^T_{[0,n-1]\cap\mathbb{Z}})$, the eventually-defined value for $\mathrm{h}_\mathsf{top}(X,T)$ does not depend on the choice of metrics $d$ giving the topology of $X$.

\subsection{How to use this definition?}
Next, as a very simple application, we demonstrate how to use this definition to show the following two well-known statements:

\subsubsection*{Topological entropy does not increase when taking factors}
Recall that by a \textbf{factor map} $\pi:(X,T)\to(Y,S)$ we precisely understand a continuous and surjective map $\pi:X\to Y$ satisfying $\pi\circ T=S\circ\pi$, where $T:X\to X$ and $S:Y\to Y$ are homeomorphisms of compact metrizable spaces $X$ and $Y$, respectively.
\begin{itemize}\item
Proposition\textbf{.}
If $\pi:(X,T)\to(Y,S)$ is a factor map, then $\mathrm{h}_\mathsf{top}(Y,S)\le\mathrm{h}_\mathsf{top}(X,T)$.
\end{itemize}
\begin{proof}
Take compatible metrics $d$ on $X$ and $\rho$ on $Y$. Since $\pi:X\to Y$ is continuous, for any $\epsilon>0$ there is some $0<\delta<\epsilon$ such that for any $x,x^\prime\in X$ if $d(x,x^\prime)\le\delta$ then $\rho(\pi(x),\pi(x^\prime))\le\epsilon$. Since $\pi:X\to Y$ is onto, one can take a map $f:Y\to X$ by sending each $y\in Y$ to some $f(y)\in\pi^{-1}(y)$ (assigned arbitrarily).

Here again note that the map $f:Y\to X$ does not have any topological property (e.g., not required to be continuous), nor any dynamical property. Although both $X$ and $Y$ are equipped with topological structures, one has to ignore them when restricting to the map $f$ defined between two \textit{sets} only.

Fix $\epsilon>0$ for the moment. For every $n\in\mathbb{N}$ take a map $g_n:X\to B_n$ such that $\#B_n=\mathrm{part}_\delta(X,d^T_{[0,n-1]\cap\mathbb{Z}})$ and that $g_n$ is a $\delta$-partition for $X$ with respect to the distance $d^T_{[0,n-1]\cap\mathbb{Z}}$ (as stated previously). Consider the map $g_n\circ f:Y\to B_n$. Clearly, $g_n\circ f$ is an $\epsilon$-partition for $Y$ with respect to the distance $\rho^S_{[0,n-1]\cap\mathbb{Z}}$ (because $\pi\circ T=S\circ\pi$, as mentioned before). This gives $\mathrm{part}_\epsilon(Y,\rho^S_{[0,n-1]\cap\mathbb{Z}})\le\#B_n=\mathrm{part}_\delta(X,d^T_{[0,n-1]\cap\mathbb{Z}})$. Since $n\in\mathbb{N}$ is arbitrary, and since $\delta\to0$ as $\epsilon\to0$, the statement follows.
\end{proof}

\subsubsection*{Topological entropy of the shift action over finitely many symbols}
For simplicity we explain how one can easily see that the topological entropy of the full shift over two symbols is equal to $\log2$. Denote by $\sigma:\{0,1\}^\mathbb{Z}\to\{0,1\}^\mathbb{Z}$ the \textbf{shift action} on the product space $\{0,1\}^\mathbb{Z}$, defined by: $(x_n)_{n\in\mathbb{Z}}\mapsto(x_{n+1})_{n\in\mathbb{Z}}$.

For convenience we take a distance $\mathsf{D}$ giving the product topology on $\{0,1\}^\mathbb{Z}$ as follows: $$\mathsf{D}((x_n)_{n\in\mathbb{Z}},(x_n^\prime)_{n\in\mathbb{Z}})=\sum_{n\in\mathbb{Z}}\frac{|x_n-x_n^\prime|}{2^{|n|}}\;,\quad\quad\quad\left((x_n)_{n\in\mathbb{Z}},(x_n^\prime)_{n\in\mathbb{Z}}\in\{0,1\}^\mathbb{Z}\right).$$

In order to see that $\mathrm{h}_\mathsf{top}(\{0,1\}^\mathbb{Z},\sigma)\le\log2$, take an arbitrary $\epsilon>0$ and fix it for the moment. Pick some $L_\epsilon\in\mathbb{N}$ such that if $(x_i)_{i\in\mathbb{Z}}$ and $(x_i^\prime)_{i\in\mathbb{Z}}$ in $\{0,1\}^\mathbb{Z}$ satisfy that $(x_i)_{i=-L_\epsilon}^{L_\epsilon}=(x_i^\prime)_{i=-L_\epsilon}^{L_\epsilon}$ then $\mathsf{D}((x_i)_{i\in\mathbb{Z}},(x_i^\prime)_{i\in\mathbb{Z}})\le\epsilon$. It follows that for any $n\in\mathbb{N}$ the finite set $\{0,1\}^{n+2L_\epsilon}$ along with the canonical projection $\mathrm{proj}\mathsf{|}_{-L_\epsilon}^{n+L_\epsilon-1}:\{0,1\}^\mathbb{Z}\to\{0,1\}^{n+2L_\epsilon}$, sending each $(x_i)_{i\in\mathbb{Z}}\in\{0,1\}^\mathbb{Z}$ to $(x_i)_{i=-L_\epsilon}^{n+L_\epsilon-1}\in\{0,1\}^{n+2L_\epsilon}$, is an $\epsilon$-partition for $\{0,1\}^\mathbb{Z}$ with respect to the distance $\mathsf{D}^\sigma_{\mathbb{Z}\cap[0,n-1]}$, and hence $$\log\mathrm{part}_\epsilon(\{0,1\}^\mathbb{Z},\mathsf{D}^\sigma_{\mathbb{Z}\cap[0,n-1]})\le\log\#\{0,1\}^{n+2L_\epsilon}=(n+2L_\epsilon)\cdot\log2,\quad\quad\forall\;n\in\mathbb{N}.$$ This gives $$\lim_{n\to\infty}\frac{\log\mathrm{part}_\epsilon(\{0,1\}^\mathbb{Z},\mathsf{D}^\sigma_{\mathbb{Z}\cap[0,n-1]})}{n}\le\log2.$$ Since $\epsilon>0$ was taken arbitrarily, one finds that $\mathrm{h}_\mathsf{top}(\{0,1\}^\mathbb{Z},\sigma)\le\log2$.

Conversely, to see that $\mathrm{h}_\mathsf{top}(\{0,1\}^\mathbb{Z},\sigma)\ge\log2$, define a map $f_n:\{0,1\}^n\to\{0,1\}^\mathbb{Z}$, for any $n\in\mathbb{N}$, by sending each $(x_i)_{i=0}^{n-1}\in\{0,1\}^n$ to $(\underline{x}_i)_{i\in\mathbb{Z}}\in\{0,1\}^\mathbb{Z}$, where $\underline{x}_i=x_i$ for every $i\in\mathbb{Z}\cap[0,n-1]$ while $\underline{x}_i=0$ for all $i\in\mathbb{Z}\setminus[0,n-1]$. If one endows the set $\{0,1\}^n$ with a distance $\mathsf{dist}_{l^\infty}$ taking values $1$ for any pair of distinct points in $\{0,1\}^n$, and $0$ otherwise (which in consequence gives the discrete topology on $\{0,1\}^n$), then the map $f_n:\{0,1\}^n\to\{0,1\}^\mathbb{Z}$ is distance-increasing with respect to the compact metric spaces $(\{0,1\}^n,\mathsf{dist}_{l^\infty})$ and $(\{0,1\}^\mathbb{Z},\mathsf{D}^\sigma_{\mathbb{Z}\cap[0,n-1]})$: $$\mathsf{dist}_{l^\infty}((x_i)_{i=0}^{n-1},(x_i^\prime)_{i=0}^{n-1})\le\mathsf{D}^\sigma_{\mathbb{Z}\cap[0,n-1]}(f_n((x_i)_{i=0}^{n-1}),f_n((x_i^\prime)_{i=0}^{n-1})),\;\,\forall(x_i)_{i=0}^{n-1},(x_i^\prime)_{i=0}^{n-1}\in\{0,1\}^n.$$ This implies that for any $\epsilon>0$ and any $n\in\mathbb{N}$$$\log\mathrm{part}_\epsilon(\{0,1\}^n,\mathsf{dist}_{l^\infty})\le\log\mathrm{part}_\epsilon(\{0,1\}^\mathbb{Z},\mathsf{D}^\sigma_{\mathbb{Z}\cap[0,n-1]}).$$ Moreover, it is clear that for any $0<\epsilon<1$ and any $n\in\mathbb{N}$ one has $$\log\mathrm{part}_\epsilon(\{0,1\}^n,\mathsf{dist}_{l^\infty})=n\cdot\log2$$ (note that $\log\mathrm{part}_\epsilon(\{0,1\}^n,\mathsf{dist}_{l^\infty})=0$ for all $\epsilon\ge1$ and $n\in\mathbb{N}$). Thus, we conclude that $\mathrm{h}_\mathsf{top}(\{0,1\}^\mathbb{Z},\sigma)\ge\log2$.

Note that here an analogue of the Lebesgue lemma (for $\epsilon$-width dimension, see the final subsection of this section) reduces immediately to an obvious and basic combinatorial statement (where the distance $\mathsf{d}_{l^\infty}$ on a finite product is defined as the maximum among all the $\mathsf{d}$-differences between two corresponding entries ranging over every coordinate of two points in the finite product):
\begin{itemize}\item
Let $A$ be a finite set of cardinality $\#A$, equipped with a distance $\mathsf{d}$ such that $\mathsf{d}(a,a^\prime)\ge1$ for any pair of distinct points $a,a^\prime\in A$. Then for every $0<\epsilon<1$ and every $n\in\mathbb{N}$ one has $\log\mathrm{part}_\epsilon(A^n,\mathsf{d}_{l^\infty})=n\cdot\log\#A$.
\end{itemize}
This idea, of course, is not essentially new at all, but provides a slightly new angel which comes to be highly parallel to mean dimension of full shifts. Also note that under the discrete topology any finite set becomes a compact metrizable space. Furthermore, from such a point of view one can eventually reformulate and unify some concrete expression for topological entropy of subshifts (over finitely many symbols), and is further allowed to see again that there is a minimal subshift of topological entropy equal to an arbitrarily assigned value within $[0,+\infty]$.

\subsection{Definition of topological entropy for flows}
Let $(\phi_t)_{t\in\mathbb{R}}$ be a flow defined on a compact metrizable space $X$. Take a (continuous) metric $d$ on $X$ compatible with its topology. For a compact subset $F\subset\mathbb{R}$ we still consider the Bowen metric on $X$ generated by $F$ under $\phi$, denoted also by $d^\phi_F$ which is defined as $$d^\phi_F(x,x^\prime)=\max_{t\in F}d(\phi_t(x),\phi_t(x^\prime)),\quad\quad x,x^\prime\in X.$$ Clearly, $d^\phi_F$ gives the topology of $X$, too. The \textbf{topological entropy of the flow $(X,(\phi_t)_{t\in\mathbb{R}})$} is (re-)defined by $$\mathrm{h}_\mathsf{top}(X,(\phi_t)_{t\in\mathbb{R}})=\lim_{\epsilon\to0}\lim_{r\to+\infty}\frac{\log\mathrm{part}_\epsilon(X,d^\phi_{[0,r]})}{r}.$$ As expected similarly, this definition leads to the same value as the original definition for topological entropy of flows. Moreover, although there is a compatible metric being involved in the defining expression, the defined value does not depend upon the choice of compatible metrics $d$ on $X$.

In the above expression, the outer limit exists because the term $\log\mathrm{part}_\epsilon(X,d^\phi_{[0,r]})$ is monotone in $\epsilon>0$. However, in relation to the existence of the inner limit, note that Fekete's lemma (cf. the Ornstein--Weiss lemma \cite[Appendix]{LW00}) is still well applied to the present situation. In fact, for every fixed $\epsilon>0$ the one-parameter, nonnegative and real-valued function $\log\mathrm{part}_\epsilon(X,d^\phi_{[0,r]})$ (with respect to the variable $r\in\mathbb{R}$) is \textit{bounded on each interval} $(0,R]$ (for any $R>0$), and hence, the proof of Fekete's lemma for sequences can be basically reproduced adapting to this case. In consequence, due to the subadditivity of the term $\log\mathrm{part}_\epsilon(X,d^\phi_{[0,r]})$ in $r\in\mathbb{R}$ the above-mentioned inner limit exists.

Note that in the definition of topological entropy using finite open covers, one encounters joins of uncountably many finite open covers for flows, while in Bowen's definition for topological entropy, the $\limsup_t$ (or $\liminf_t$) of a one-parameter real-valued function $a(t):\mathbb{R}\to\mathbb{R}$ as $t$ tends to $+\infty$ along the real line, does not necessarily agree with the $\limsup_n$ of $a(n):\mathbb{Z}\to\mathbb{R}$ as $n$ goes to $+\infty$ within $\mathbb{Z}$. Now the modified definition uses the Bowen metrics and finite partitions (instead of finite open covers) of the state space. In general, such an approach to the definition of topological entropy for one-parameter actions applies directly to all the (discrete or continuous) amenable group actions (and also to conditional or relative topological entropy).

In what follows we consider the intermediate limit for iterates of flows. Let $X$ be a compact metrizable space. For a flow $\phi=(\phi_t)_{t\in\mathbb{R}}$ defined on $X$ and $N\in\mathbb{R}$ the notation $\phi^N$ is to denote the flow $(\phi_{Nt})_{t\in\mathbb{R}}$ (namely the $N$-iterate) defined on $X$. Note that for topological entropy one has $\mathrm{h}_\mathsf{top}(X,\phi^N)=N\cdot\mathrm{h}_\mathsf{top}(X,\phi)$. But occasionally this would not be adequate for some purpose. It turns out that in the modified approach to the definition of topological entropy this indeed adapts easily to the intermediate limit.
\begin{lemma}
For any fixed $\epsilon>0$ and any compatible metric $d$ on $X$ one has $$\lim_{r\to+\infty}\frac{\log\mathrm{part}_\epsilon(X,d^{\phi^N}_{[0,r]})}{r}=N\cdot\lim_{r\to+\infty}\frac{\log\mathrm{part}_\epsilon(X,d^\phi_{[0,r]})}{r}$$ for every $N\in\mathbb{R}$.
\end{lemma}
\begin{proof}
This follows simply from its definition and the existence of the limit (by replacing $r$ with $Nr$ in a limit).
\end{proof}
This lemma will not be used in our main proof, as the treatment there is more straightforward. Nevertheless, it looks more satisfactory, and from this point of view one may see that the existence of the intermediate limit gives rise to an extra advantage, possibly being more flexible over estimates and calculation.

In a classical handling of topological entropy for flows, one usually first defines that for a homeomorphism $T:X\to X$ of a compact metrizable space $X$, and then extends from a $\mathbb{Z}$-action (i.e., generated by a homeomorphism) to an $\mathbb{R}$-action (i.e., a flow) by relating to its time-$1$ map. More precisely, the topological entropy of a flow $(X,(\phi_t)_{t\in\mathbb{R}})$ is naturally defined to be that of $(X,\phi_1)$. Now one may find an angle on the reason to see again why it suffices. In fact, by the alternative, using the continuity of the flow action $\Phi:\mathbb{R}\times X\to X$ on the compact subset $[0,1]\times X$, along with the existence of the intermediate limit as pointed out previously, one has actually \textit{proved} that the topological entropy of a flow $(X,(\phi_t)_{t\in\mathbb{R}})$ is equal to the one of the homeomorphism given by its time-$1$ map $\phi_1:X\to X$.

\subsection{Review of mean dimension}
In this subsection we recall the definition of mean dimension for homeomorphisms (i.e., $\mathbb{Z}$-actions) and flows (i.e., $\mathbb{R}$-actions). For details the reader is referred to \cite{LW00}.

Let $(X,d_X)$ be a compact metric space. Let $\epsilon>0$. Define the \textbf{$\epsilon$-width dimension for $X$ with respect to the distance $d_X$}, denoted by $\mathrm{Widim}_\epsilon(X,d_X)$, as the minimum Lebesgue covering dimension (which is often referred to as the topological dimension) $\mathrm{dim}(P)$ of any compact metrizable space $P$ such that there is a continuous map $f:X\to P$ satisfying that $f(x)=f(x^\prime)$ implies $d_X(x,x^\prime)\le\epsilon$, for any $x,x^\prime\in X$.

Note that for any compact metric space $(X,d_X)$ and $\epsilon>0$, $\mathrm{Widim}_\epsilon(X,d_X)$ is always bounded from above by $\mathrm{dim}(X)$. Moreover, the topological dimension of $X$ can be recovered by $\mathrm{dim}(X)=\lim_{\epsilon\to0}\mathrm{Widim}_\epsilon(X,d_X)$ which does not depend upon the choice of the metrics $d_X$ compatible with the topology of $X$ (and $d_X$ hence is no longer included in the notation $\mathrm{dim}(X)$).

For compact metric spaces $(X_1,d_{X_1})$ and $(X_2,d_{X_2})$, the $l^\infty$-distance $\mathrm{dist}_{l^\infty}$, which is also denoted by $d_{X_1}\times_{l^\infty}d_{X_2}$, on the product space $X_1\times X_2$ is defined by $$\mathrm{dist}_{l^\infty}((x_1,x_2),(x_1^\prime,x_2^\prime))=\max\{d_{X_1}(x_1,x_1^\prime),d_{X_2}(x_2,x_2^\prime)\}$$ for $(x_1,x_2),(x_1^\prime,x_2^\prime)\in X_1\times X_2$. The distance $\mathrm{dist}_{l^\infty}$ gives the product topology of $X_1\times X_2$. For $\epsilon>0$ it is clear that under this distance $$\mathrm{Widim}_\epsilon(X_1\times X_2,d_{X_1}\times_{l^\infty}d_{X_2})\le\mathrm{Widim}_\epsilon(X_1,d_{X_1})+\mathrm{Widim}_\epsilon(X_2,d_{X_2}).$$ The above-mentioned notion and notation apply directly to finitely many members. The number of the participants in the product is not precisely indicated in the notation of the distance $\mathrm{dist}_{l^\infty}$ because this does not cause any confusion.

The famous Lebesgue lemma, for $\epsilon$-width dimension, is formulated as follows.
\begin{lemma}
For any $0<\epsilon<1$ and $n\in\mathbb{N}$ one has $\mathrm{Widim}_\epsilon([0,1]^n,\mathsf{dist}_{l^\infty})=n$.
\end{lemma}

Now we are involving dynamics. Let $X$ be a compact metrizable space. Let $T:X\to X$ be a homeomorphism of $X$ and $(\phi_t)_{t\in\mathbb{R}}$ a flow defined on $X$. Take any distance $d$ on $X$, which gives the topology of $X$. The \textbf{mean dimension} of $(X,T)$ is defined by $$\mathrm{mdim}(X,T)=\lim_{\epsilon\to0}\lim_{n\to\infty}\frac{\mathrm{Widim}_\epsilon(X,d^T_{[0,n-1]\cap\mathbb{Z}})}{n},$$ and the mean dimension of $(X,(\phi_t)_{t\in\mathbb{R}})$ is defined by $$\mathrm{mdim}(X,(\phi_t)_{t\in\mathbb{R}})=\lim_{\epsilon\to0}\lim_{r\to+\infty}\frac{\mathrm{Widim}_\epsilon(X,d^\phi_{[0,r]})}{r}.$$ In each of the above two expressions, the inner limit exists because $\mathrm{Widim}_\epsilon(X,d^T_{[0,n-1]\cap\mathbb{Z}})$ (or $\mathrm{Widim}_\epsilon(X,d^\phi_{[0,r]})$) is subadditive in $n$ (or $r$), while the outer limit exists because $\mathrm{Widim}_\epsilon(X,d)$ is monotone in $\epsilon$.

Similar to topological entropy, mean dimension is a topological invariant of dynamical systems. More precisely, although a compatible metric is involved in the expressions for the definition of mean dimension, the defined values do not depend on the choice of compatible metrics $d$ on $X$. Furthermore, it is clear that the mean dimension of a flow is equal to that of the homeomorphism given by its time-$1$ map $\phi_1:X\to X$.

\section{Weakly equivalent flows without fixed points:\\mean dimension relation}
\subsection{Statement of the theorem}
First of all we state the main result of this section:
\begin{theorem}
If $(X,(\phi_t)_{t\in\mathbb{R}})$ and $(Y,(\psi_s)_{s\in\mathbb{R}})$ are weakly equivalent flows without fixed points, then $$m\cdot\mathrm{mdim}(Y,(\psi_s)_{s\in\mathbb{R}})\le\mathrm{mdim}(X,(\phi_t)_{t\in\mathbb{R}})\le M\cdot\mathrm{mdim}(Y,(\psi_s)_{s\in\mathbb{R}}),$$ where $0<m\le M<+\infty$ are computable in the sense of Section 2; in particular, $\mathrm{mdim}(X,(\phi_t)_{t\in\mathbb{R}})=C\cdot\mathrm{mdim}(Y,(\psi_s)_{s\in\mathbb{R}})$, where $C$ is a finite and positive number.
\end{theorem}

More precisely, take metrics $d$ and $\rho$ giving the topology on $X$ and $Y$, respectively. Let $\pi:X\to Y$ be a homeomorphism giving the weak equivalence relation between $(X,(\phi_t)_{t\in\mathbb{R}})$ and $(Y,(\psi_s)_{s\in\mathbb{R}})$. Let $\theta:\mathbb{R}\times X\to\mathbb{R}$ be the map defined by $\pi:X\to Y$, as mentioned precisely in Section 2. Put $m=\min_{x\in X}\theta(1,x)$ and $M=\max_{x\in X}\theta(1,x)$. We shall prove that $(X,(\phi_t)_{t\in\mathbb{R}})$ and $(Y,(\psi_s)_{s\in\mathbb{R}})$ share the mean dimension relation: $$m\cdot\mathrm{mdim}(Y,(\psi_s)_{s\in\mathbb{R}})\le\mathrm{mdim}(X,(\phi_t)_{t\in\mathbb{R}})\le M\cdot\mathrm{mdim}(Y,(\psi_s)_{s\in\mathbb{R}}).$$

\subsection{Estimate for $M$}
First we prepare a simple lemma, which is needed in the subsequent sections, too.
\begin{lemma}\label{ttx}
Let $(X,(\phi_t)_{t\in\mathbb{R}})$ be a flow. Suppose that $d$ is a metric that gives the topology on $X$. Then for each $\epsilon>0$ there exists some $\delta>0$ depending on $\epsilon$ such that for any $t,t^\prime\in\mathbb{R}$ with $|t-t^\prime|<\delta$ one has $d(\phi_t(x),\phi_{t^\prime}(x))<\epsilon$ for all $x\in X$.
\end{lemma}
\begin{proof}
For an $\epsilon>0$ one finds such a $0<\delta<1$ by the compactness of $[0,1]\times X$, i.e., such that if $0\le t,t^\prime\le1$ satisfy $|t-t^\prime|<\delta$ then one has $d(\phi_t(x),\phi_{t^\prime}(x))<\epsilon$ for all $x\in X$. Note that $d(\phi_t(x),\phi_{t^\prime}(x))=d(\phi_{t-u}(\phi_u(x)),\phi_{t^\prime-u}(\phi_u(x)))$ for all $x\in X$, where $u$ is such that both $t-u$ and $t^\prime-u$ belong to $[0,1]$. It follows immediately that this $\delta>0$ applies to the whole range $\mathbb{R}$, namely not only restricted to the condition $0\le t,t^\prime\le1$.
\end{proof}

Now we are ready to show that $\mathrm{mdim}(X,(\phi_t)_{t\in\mathbb{R}})\le M\cdot\mathrm{mdim}(Y,(\psi_s)_{s\in\mathbb{R}})$.

Fix $\epsilon>0$ arbitrarily. Since $X$ is compact and since $\pi:X\to Y$ is a homeomorphism, there exists $0<\delta<\epsilon$ such that $\rho(\pi(x),\pi(x^\prime))\le\delta$ implies $d(x,x^\prime)\le\epsilon$, for any $x,x^\prime\in X$. Applying Lemma \ref{ttx} to the flow $(Y,(\psi_s)_{s\in\mathbb{R}})$, there is some $0<\eta<\delta/2$ such that if $s,s^\prime\in\mathbb{R}$ satisfy $|s-s^\prime|\le\eta$, then $\rho(\psi_s(y),\psi_{s^\prime}(y))\le\delta/2$ for all $y\in Y$.

For any $n\in\mathbb{N}$ consider the product space $Y\times[0,Mn]^n$ equipped with the $l^\infty$-distance $\rho^\psi_{[0,Mn]}\times_{l^\infty}\mathrm{dist}_{l^\infty}$, where the standard $n$-dimensional cube $[0,Mn]^n$ is endowed typically with the $l^\infty$-distance $\mathrm{dist}_{l^\infty}$.

Note that Lemma \ref{mM} is being used implicitly in the proof.

\begin{lemma}
$$\mathrm{Widim}_\epsilon(X,d^\phi_{[0,n-1]\cap\mathbb{Z}})\le\mathrm{Widim}_\eta(Y\times[0,Mn]^n,\rho^\psi_{[0,Mn]}\times_{l^\infty}\mathrm{dist}_{l^\infty}).$$
\end{lemma}
\begin{proof}
Let $g:X\to Y\times[0,Mn]^n$ be a map defined by $$x\mapsto(\pi(x),(\theta(i,x))_{i=1}^n).$$ Clearly, the map $g$ is continuous. Take a compact metrizable space $P$ with $\mathrm{dim}(P)=\mathrm{Widim}_\eta(Y\times[0,Mn]^n,\rho^\psi_{[0,Mn]}\times_{l^\infty}\mathrm{dist}_{l^\infty})$ and a continuous map $f:Y\times[0,Mn]^n\to P$ such that if $y,y^\prime\in Y$ and $(l_i)_{i=1}^n,(l^\prime_i)_{i=1}^n\in[0,Mn]^n$ satisfy $f((y,(l_i)_{i=1}^n))=f((y^\prime,(l^\prime_i)_{i=1}^n))$ then $\rho^\psi_{[0,Mn]}\times_{l^\infty}\mathrm{dist}_{l^\infty}((y,(l_i)_{i=1}^n),(y^\prime,(l^\prime_i)_{i=1}^n))\le\eta$.

Now consider the continuous map $f\circ g:X\to P$. Suppose that $x,x^\prime\in X$ satisfy $f(g(x))=f(g(x^\prime))$. By the choice of $f$ one has $\rho^\psi_{[0,Mn]}(\pi(x),\pi(x^\prime))\le\eta<\delta/2$ and $|\theta(i,x)-\theta(i,x^\prime)|\le\eta$, for every integer $1\le i\le n$. Since $$0=\theta(0,x)<\theta(1,x)<\cdots<\theta(i,x)<\cdots<\theta(n,x)\le Mn,$$ one has $$\rho(\psi_{\theta(i,x)}(\pi(x)),\psi_{\theta(i,x)}(\pi(x^\prime)))\le\delta/2,\quad\quad\forall i\in[0,n]\cap\mathbb{Z}.$$ Note that by the choice of $\eta$, $$\rho(\psi_{\theta(i,x)}(\pi(x^\prime)),\psi_{\theta(i,x^\prime)}(\pi(x^\prime)))\le\delta/2,\quad\quad\forall i\in[0,n]\cap\mathbb{Z}.$$ Hence, $$\rho(\psi_{\theta(i,x)}(\pi(x)),\psi_{\theta(i,x^\prime)}(\pi(x^\prime)))\le\delta,\quad\quad\forall i\in[0,n]\cap\mathbb{Z},$$ and it follows that $$\rho(\pi(\phi_i(x)),\pi(\phi_i(x^\prime)))\le\delta,\quad\quad\forall i\in[0,n]\cap\mathbb{Z}.$$ By the choice of $\delta$ we deduce that $$d(\phi_i(x),\phi_i(x^\prime))\le\epsilon,\quad\quad\forall i\in[0,n]\cap\mathbb{Z}.$$ This implies that $$\mathrm{Widim}_\epsilon(X,d^\phi_{[0,n]\cap\mathbb{Z}})\le\mathrm{dim}(P).$$ Thus, the statement follows.
\end{proof}

By the above lemma, $$\mathrm{Widim}_\epsilon(X,d^\phi_{[0,n-1]\cap\mathbb{Z}})\le\mathrm{Widim}_\eta(Y,\rho^\psi_{[0,Mn]})+n.$$ Since this is true for every $n\in\mathbb{N}$, and since $\eta\to0$ as $\epsilon\to0$, one has $$\mathrm{mdim}(X,(\phi_t)_{t\in\mathbb{R}})\le M\cdot\mathrm{mdim}(Y,(\psi_s)_{s\in\mathbb{R}})+1.$$ Note that for any fixed $N\in\mathbb{N}$ this is still valid for $(X,(\phi_{Nt})_{t\in\mathbb{R}})$ and $(Y,(\psi_{Ns})_{s\in\mathbb{R}})$ (namely the $N$-iterates) along the common map $\theta:\mathbb{R}\times X\to\mathbb{R}$ with respect to the flows $(\phi_t)_{t\in\mathbb{R}}$ and $(\psi_s)_{s\in\mathbb{R}}$ defined on $X$ and $Y$, respectively, which are weakly equivalent under $\pi:X\to Y$, i.e., the constant $M$ does not depend upon $N\in\mathbb{N}$. It follows that $$\mathrm{mdim}(X,(\phi_t)_{t\in\mathbb{R}})\le M\cdot\mathrm{mdim}(Y,(\psi_s)_{s\in\mathbb{R}})+\frac1N,$$ for all $N\in\mathbb{N}$. Therefore, $$\mathrm{mdim}(X,(\phi_t)_{t\in\mathbb{R}})\le M\cdot\mathrm{mdim}(Y,(\psi_s)_{s\in\mathbb{R}}).$$

\subsection{Two continuity lemmas}
In this subsection, we prepare two extra lemmas in relation to the map $\theta:\mathbb{R}\times X\to\mathbb{R}$, which is being used later.

Note that we do not have the uniform continuity of $\theta:\mathbb{R}\times X\to\mathbb{R}$. However, we have
\begin{lemma}\label{thetattx}
For any $\delta>0$ there exists $0<\eta<\delta$ such that if $t,t^\prime\in\mathbb{R}$ satisfy $|t-t^\prime|<\eta$, then $|\theta(t,x)-\theta(t^\prime,x)|<\delta$ for all $x\in X$.
\end{lemma}
\begin{proof}
Fix $\delta>0$. The uniform continuity of $\theta$ on $[0,\delta]\times X$ (being compact) implies the existence of such an $0<\eta<\delta$ valid for $[0,\delta]\times X$, i.e., satisfying that if $0\le t,t^\prime\le\delta$ satisfy $|t-t^\prime|<\eta$, then $|\theta(t,x)-\theta(t^\prime,x)|<\delta$ for all $x\in X$. To see that this is still true for $\theta:\mathbb{R}\times X\to\mathbb{R}$, note that $$\theta(t,x)-\theta(t^\prime,x)=\theta(t-s,\phi_s(x))-\theta(t^\prime-s,\phi_s(x)).$$ Therefore the desired statement can be derived from the fact $|(t-s)-(t^\prime-s)|=|t-t^\prime|<\eta$ with $0\le t-s,t^\prime-s\le\delta$ by taking a proper $s\in\mathbb{R}$.
\end{proof}

Let us turn to the next lemma. Take an arbitrary $x\in X$ and fix it temporarily. By Lemma \ref{mM}, $\theta(n,x)\ge mn$, for all $n\in\mathbb{N}$. Note that $\theta(0,x)=0$, and that $\theta(t,x)$ is continuous and strictly increases in $t\in\mathbb{R}$. It follows that there is a sequence $\{t_x(n)\}_{n=0}^{+\infty}\subset\mathbb{R}$, which is uniquely determined by $x\in X$, such that $\theta(t_x(n),x)=mn$ for every integer $n\ge0$. Note that $t_x(0)=0$ and $0\le t_x(n)\le n$, for every integer $n\ge0$.

Now limit our attention to the map $t_\cdot(n):X\to\mathbb{R}$, $x\mapsto t_x(n)$, for each fixed integer $n\ge0$.
\begin{lemma}\label{txn}
For each $n\in\mathbb{N}$, $t_x(n)$ is continuous in $x\in X$.
\end{lemma}
\begin{proof}
Fix $n\in\mathbb{N}$. Argue by contradiction. Suppose that a sequence $\{x_i\}_{i=1}^\infty\subset X$ converges to some $x\in X$ (as $i\to\infty$), while the sequence $\{t_{x_i}(n)\}_{i=1}^\infty\subset\mathbb{R}$ does not converge to $t_x(n)$ (as $i\to\infty$). Without loss of generality, one may assume that $t_{x_i}(n)\to t^\ast\ne t_x(n)$ as $i\to\infty$. Since $\theta(t,x)$ is continuous in $(t,x)\in\mathbb{R}\times X$, $$\theta(t^\ast,x)=\lim_{i\to\infty}\theta(t_{x_i}(n),x_i)=mn=\theta(t_x(n),x).$$ This contradicts the property that $\theta(t,x)$ is strictly increasing in $t\in\mathbb{R}$.
\end{proof}

\subsection{Estimate for $m$}
With the above preparation, next we are ready to prove $m\cdot\mathrm{mdim}(Y,(\psi_s)_{s\in\mathbb{R}})\le\mathrm{mdim}(X,(\phi_t)_{t\in\mathbb{R}})$.

Take an arbitrary $\epsilon>0$ and fix it for the moment. By the continuity of $\pi:X\to Y$ with the compactness of $X$, and by Lemma \ref{ttx}, there is some $0<\delta<\epsilon$ satisfying the following two conditions:
\begin{itemize}
\item
if $x,x^\prime\in X$ satisfy $d(x,x^\prime)\le\delta$, then $\rho(\pi(x),\pi(x^\prime))\le\epsilon/2$;
\item
if $s,s^\prime\in\mathbb{R}$ satisfy $|s-s^\prime|\le\delta$, then $\rho(\psi_s(y),\psi_{s^\prime}(y))\le\epsilon/2$ for all $y\in Y$.
\end{itemize}
By Lemma \ref{thetattx}, one can find some $0<\eta<\delta$ such that if $t,t^\prime\in\mathbb{R}$ satisfy $|t-t^\prime|\le\eta$ then $|\theta(t,x)-\theta(t^\prime,x)|\le\delta$ for all $x\in X$.

For any $n\in\mathbb{N}$ consider the product space $X\times[0,n]^n$ equipped with the $l^\infty$-distance $d^\phi_{[0,n]}\times_{l^\infty}\mathrm{dist}_{l^\infty}$, where the standard $n$-dimensional cube $[0,n]^n$ is endowed typically with the $l^\infty$-distance $\mathrm{dist}_{l^\infty}$.

Note that Lemma \ref{txn}, in company with the notations introduced there, is being used implicitly in the proof of the lemma below. Also note that for any fixed $n\in\mathbb{N}$$$0=t_x(0)<t_x(1)<\cdots<t_x(n)\le n$$ holds for every $x\in X$.

\begin{lemma}
$$\mathrm{Widim}_\epsilon(Y,\rho^\psi_{[0,mn]\cap(m\mathbb{Z})})\le\mathrm{Widim}_\eta(X\times[0,n]^n,d^\phi_{[0,n]}\times_{l^\infty}\mathrm{dist}_{l^\infty}).$$
\end{lemma}
\begin{proof}
Let $g:Y\to X\times[0,n]^n$ be a map defined by $$y\mapsto(x,(t_x(i))_{i=1}^n),$$ where $x\in X$ is to denote $\pi^{-1}(y)$ for $y\in Y$, and $x^\prime\in X$ is planned to denote $\pi^{-1}(y^\prime)$ for $y^\prime\in Y$ in the sequel. It is clear that the map $g$ is continuous. Take a compact metrizable space $P$ with $\mathrm{dim}(P)=\mathrm{Widim}_\eta(X\times[0,n]^n,d^\phi_{[0,n]}\times_{l^\infty}\mathrm{dist}_{l^\infty})$ and a continuous map $f:X\times[0,n]^n\to P$ such that if $x,x^\prime\in X$ and $(l_i)_{i=1}^n,(l^\prime_i)_{i=1}^n\in[0,n]^n$ satisfy $f((x,(l_i)_{i=1}^n))=f((x^\prime,(l^\prime_i)_{i=1}^n))$ then one has both $d^\phi_{[0,n]}(x,x^\prime)\le\eta$ and $\mathrm{dist}_{l^\infty}((l_i)_{i=1}^n,(l^\prime_i)_{i=1}^n)\le\eta$.

Now consider the continuous map $f\circ g:Y\to P$. Suppose that $y,y^\prime\in Y$ satisfy $f(g(y))=f(g(y^\prime))$. By the choice of $f$ one has $d^\phi_{[0,n]}(x,x^\prime)\le\eta<\delta$ and $|t_x(i)-t_{x^\prime}(i)|\le\eta$, for every integer $0\le i\le n$. Since $t_x(i)\in[0,n]$ for all $i\in[0,n]\cap\mathbb{Z}$, one has $$d(\phi_{t_x(i)}(x),\phi_{t_x(i)}(x^\prime))<\delta,\quad\quad\forall i\in[0,n]\cap\mathbb{Z}.$$ By the choice of $\delta$, $$\rho(\pi(\phi_{t_x(i)}(x)),\pi(\phi_{t_x(i)}(x^\prime)))\le\epsilon/2,\quad\quad\forall i\in[0,n]\cap\mathbb{Z}.$$ By the definition of $\theta$, $$\rho(\psi_{\theta(t_x(i),x)}(\pi(x)),\psi_{\theta(t_x(i),x^\prime)}(\pi(x^\prime)))\le\epsilon/2,\quad\quad\forall i\in[0,n]\cap\mathbb{Z}.$$ By the definition of $t_x(i)$, $$\rho(\psi_{mi}(\pi(x)),\psi_{\theta(t_x(i),x^\prime)}(\pi(x^\prime)))\le\epsilon/2,\quad\quad\forall i\in[0,n]\cap\mathbb{Z}.$$ Note that $mi$ is also represented as $mi=\theta(t_{x^\prime}(i),x^\prime)$, for $i\in[0,n]\cap\mathbb{Z}$. It follows from the choice of $\eta$ that $|\theta(t_x(i),x^\prime)-\theta(t_{x^\prime}(i),x^\prime)|\le\delta$ which implies that $$|\theta(t_x(i),x^\prime)-mi|\le\delta,\quad\quad\forall i\in[0,n]\cap\mathbb{Z}.$$ By the choice of $\delta$ we deduce that $$\rho(\psi_{\theta(t_x(i),x^\prime)}(\pi(x^\prime)),\psi_{mi}(\pi(x^\prime)))\le\epsilon/2,\quad\quad\forall i\in[0,n]\cap\mathbb{Z}.$$ Therefore, $$\rho(\psi_{mi}(\pi(x)),\psi_{mi}(\pi(x^\prime)))\le\epsilon,\quad\quad\forall i\in[0,n]\cap\mathbb{Z}$$ and hence $\rho^\psi_{[0,mn]\cap(m\mathbb{Z})}(y,y^\prime)\le\epsilon$. This implies that $$\mathrm{Widim}_\epsilon(Y,\rho^\psi_{[0,mn]\cap(m\mathbb{Z})})\le\mathrm{dim}(P).$$ Thus, the statement follows.
\end{proof}

By the above lemma, for any $n\in\mathbb{N}$$$\mathrm{Widim}_\epsilon(Y,\rho^\psi_{[0,mn]\cap(m\mathbb{Z})})\le\mathrm{Widim}_\eta(X,d^\phi_{[0,n]})+n.$$ Since $\eta\to0$ as $\epsilon\to0$, this implies that $$m\cdot\mathrm{mdim}(Y,(\psi_s)_{s\in\mathbb{R}})\le\mathrm{mdim}(X,(\phi_t)_{t\in\mathbb{R}})+1.$$ By reason of the same proof as in the estimate for $M$, one has $$m\cdot\mathrm{mdim}(Y,(\psi_s)_{s\in\mathbb{R}})\le\mathrm{mdim}(X,(\phi_t)_{t\in\mathbb{R}}).$$ Thus, we conclude.

\section{Weakly equivalent flows without fixed points:\\topological entropy relation}
The main purpose of this section is to re-establish Ohno's theorem for topological entropy of weakly equivalent flows without fixed points with a substantially different method which comes to be much more direct and elementary.

\begin{theorem}
If $(X,(\phi_t)_{t\in\mathbb{R}})$ and $(Y,(\psi_s)_{s\in\mathbb{R}})$ are weakly equivalent flows without fixed points, then $$m\cdot\mathrm{h}_\mathsf{top}(Y,(\psi_s)_{s\in\mathbb{R}})\le\mathrm{h}_\mathsf{top}(X,(\phi_t)_{t\in\mathbb{R}})\le M\cdot\mathrm{h}_\mathsf{top}(Y,(\psi_s)_{s\in\mathbb{R}}),$$ where $0<m\le M<+\infty$ are computable in the sense of Section 2; in particular, $\mathrm{h}_\mathsf{top}(X,(\phi_t)_{t\in\mathbb{R}})=C\cdot\mathrm{h}_\mathsf{top}(Y,(\psi_s)_{s\in\mathbb{R}})$, where $C$ is a finite and positive number.
\end{theorem}

More specifically, take compatible metrics $d$ and $\rho$ on $X$ and $Y$, respectively. Let $\pi:X\to Y$ be a homeomorphism giving the weak equivalence relation between $(X,(\phi_t)_{t\in\mathbb{R}})$ and $(Y,(\psi_s)_{s\in\mathbb{R}})$. Let $\theta:\mathbb{R}\times X\to\mathbb{R}$ be the map induced by $\pi:X\to Y$, as described specifically in Section 2. Put $m=\min_{x\in X}\theta(1,x)$ and $M=\max_{x\in X}\theta(1,x)$. We will prove that $(X,(\phi_t)_{t\in\mathbb{R}})$ and $(Y,(\psi_s)_{s\in\mathbb{R}})$ share the topological entropy relation: $$m\cdot\mathrm{h}_\mathsf{top}(Y,(\psi_s)_{s\in\mathbb{R}})\le\mathrm{h}_\mathsf{top}(X,(\phi_t)_{t\in\mathbb{R}})\le M\cdot\mathrm{h}_\mathsf{top}(Y,(\psi_s)_{s\in\mathbb{R}}).$$

\subsection{An elementary counting lemma}
Note that in the estimates for topological entropy, we do not need some topological structure, according to our definition, equipped on the target set, nor some continuity lemma for the map of partition. However, we have to deal with the growth rate, in a much more careful approach, in relation to the number of increasing sequences of integers. In order to put Ohno's theorem into our framework, we need prepare an elementary combinatorial counting lemma in this subsection, which is quite simple but is being essentially used later, where we will see exactly how the modified approach to the definition of topological entropy for both discrete and continuous (amenable) group actions gives rise to an extra convenience such that we are allowed to reproduce Ohno's result with a highly unified method compared to mean dimension.

\begin{lemma}\label{counting}
Let $L\in\mathbb{N}$. For any $n,N\in\mathbb{N}$ put $$A_{n,N}=\{(a_j)_{j=1}^n\in[0,LNn]^n\cap\mathbb{Z}^n:a_1\le\cdots\le a_n\}.$$ Then one has $$\lim_{N\to\infty}\left(\frac1N\cdot\lim_{n\to\infty}\frac{\log\#A_{n,N}}{n}\right)=0.$$
\end{lemma}
\begin{proof}
Fix $L\in\mathbb{N}$. Take $n,N\in\mathbb{N}$ arbitrarily and fix them temporarily. For each $k\in[0,LNn]\cap\mathbb{Z}$ put $$A_{n,N}^k=\{(a_j)_{j=1}^n\in A_{n,N}:a_n=k\}.$$ Clearly, $\#(A_{n,N}^k)\le\#(A_{n,N}^{LNn})$ for every $k\in[0,LNn]\cap\mathbb{Z}$. This implies that $\#A_{n,N}$ is bounded from above by $(LNn+1)\cdot\#(A_{n,N}^{LNn})$, namely, it suffices to show that $$\lim_{N\to\infty}\lim_{n\to\infty}\frac{\log\#(A_{n,N}^{LNn})}{nN}=0.$$

To see this, an easy but important observation is that for any $n,N\in\mathbb{N}$, the number $\#(A_{n,N}^{LNn})$ is equal to the number of solutions of nonnegative integers $b_1,\dots,b_n$ to the equation $$b_1+\cdots+b_n=LNn,$$ which is equivalent to the number of solutions of positive integers $b_1,\dots,b_n$ to the equation $$b_1+\cdots+b_n=(LN+1)\cdot n.$$

Due to the above-mentioned fact, it follows that $$\#(A_{n,N}^{LNn})={(LN+1)\cdot n-1\choose n-1}=\frac{((LN+1)\cdot n-1)!}{(n-1)!(LNn)!}\le\frac{((LN+1)n)!}{n!(LNn)!}.$$ By the Stirling approximation formula,
\begin{align*}
\lim_{n\to\infty}\frac{\log\#(A_{n,N}^{LNn})}{n}
&=\lim_{n\to\infty}\left(\frac1n\cdot\log\left(\frac{\sqrt{2\pi(LN+1)n}\cdot(\frac{(LN+1)n}{e})^{(LN+1)n}}{\sqrt{2\pi n}\cdot(\frac{n}{e})^n\cdot\sqrt{2\pi LNn}\cdot(\frac{LNn}{e})^{LNn}}\right)\right)\\
&=\lim_{n\to\infty}\log\left(\frac{(\frac{(LN+1)n}{e})^{LN+1}}{\frac{n}{e}\cdot(\frac{LNn}{e})^{LN}}\right)\\
&=\log\frac{(LN+1)^{LN+1}}{(LN)^{LN}}\\
&=\log(LN+1)+LN\cdot\log\left(1+\frac{1}{LN}\right)
\end{align*}
and in consequence
$$\lim_{N\to\infty}\left(\frac1N\cdot\lim_{n\to\infty}\frac{\log\#(A_{n,N}^{LNn})}{n}\right)=\lim_{N\to\infty}\left(\frac{\log(LN+1)}{N}+L\cdot\log\left(1+\frac{1}{LN}\right)\right)=0.$$
This proves the counting lemma.
\end{proof}

\subsection{Estimate for $m$}
First we prove $m\cdot\mathrm{h}_\mathsf{top}(Y,(\psi_s)_{s\in\mathbb{R}})\le\mathrm{h}_\mathsf{top}(X,(\phi_t)_{t\in\mathbb{R}})$.

Take an arbitrary $\epsilon>0$ and fix it for the moment. Applying Lemma \ref{ttx} to the flow $(Y,(\psi_s)_{s\in\mathbb{R}})$, there exists some $0<\eta<\epsilon$ with $1/\eta\in\mathbb{Z}$ such that if $s,s^\prime\in\mathbb{R}$ satisfy $|s-s^\prime|\le3\eta$ then $\rho(\psi_s(y),\psi_{s^\prime}(y))\le\epsilon/2$ for all $y\in Y$. By the uniform continuity of $\theta:\mathbb{R}\times X\to\mathbb{R}$ on the compact subset $[0,1]\times X$ and by the uniform continuity of the homeomorphism $\pi:X\to Y$ one can find some $0<\delta<\eta$ such that if $x,x^\prime\in X$ are such that $d(x,x^\prime)\le\delta$ then they satisfy both $\rho(\pi(x),\pi(x^\prime))\le\epsilon/2$ and $\max_{r\in[0,1]}|\theta(r,x)-\theta(r,x^\prime)|\le\eta$.

For any $n\in\mathbb{N}$ take a finite subset $I_{\eta,n}$ of $[0,Mn]$ as $$I_{\eta,n}=\left\{j\eta:j\in\left[0,\frac{Mn}{\eta}\right]\cap\mathbb{Z}\right\}.$$ For any $n,N\in\mathbb{N}$ consider the set $$A_{\eta,n,N}=\{(a_i)_{i=1}^n\in I_{\eta,Nn}^n:a_1\le\cdots\le a_n\},$$ where $I_{\eta,Nn}^n$ is to denote the self-product (of the finite set $I_{\eta,Nn}$) of $n$-times.

Now we wish to induce a factor of $(X,(\phi_t)_{t\in\mathbb{R}})$ so as not only to dominate the topological entropy naturally but also to capture needed information sufficiently. To do this, we first regard $(\theta(1,\phi_t(x)))_{t\in\mathbb{R}}$, for each $x\in X$, as a (one-parameter and real-valued) bounded continuous function (in $t$), and further, note that under the $\mathbb{R}$-shift it becomes a factor of $(X,(\phi_t)_{t\in\mathbb{R}})$, which is denoted by $(S,(\sigma_t)_{t\in\mathbb{R}})$, more precisely, defined by: $$S=\{(\theta(1,\phi_t(x)))_{t\in\mathbb{R}}:x\in X\},$$$$\sigma_r:S\to S,\quad\;(z_t)_{t\in\mathbb{R}}\mapsto(z_{t+r})_{t\in\mathbb{R}},\quad\quad\;\forall\,r\in\mathbb{R}.$$ A distance $\mathsf{D}$ on $S$ compatible with its topology is given by: $$\mathsf{D}((\theta(1,\phi_t(x)))_{t\in\mathbb{R}},(\theta(1,\phi_t(x^\prime)))_{t\in\mathbb{R}})\quad\quad\quad$$$$\quad\quad\quad=\sum_{k=0}^{+\infty}\frac{\max_{t\in[-k,k]}|\theta(1,\phi_t(x))-\theta(1,\phi_t(x^\prime))|}{2^k}\;$$$$\quad\quad\quad\quad\quad\quad\left((\theta(1,\phi_t(x)))_{t\in\mathbb{R}},(\theta(1,\phi_t(x^\prime)))_{t\in\mathbb{R}}\in S\right).$$ Next we turn to considering the map $$h:X\to X\times S,\quad\quad\quad\quad\;x\mapsto(x,(\theta(1,\phi_t(x)))_{t\in\mathbb{R}})$$ and a flow $((\phi\times\sigma)_t)_{t\in\mathbb{R}}$ on $X\times S$ defined by $$(\phi\times\sigma)_t=\phi_t\times\sigma_t:X\times S\to X\times S,\quad\,(x,z)\mapsto(\phi_t(x),\sigma_t(z)),\quad\,\quad\,\forall\;t\in\mathbb{R}.$$ As usual, a compatible metric on the product space $X\times S$ is given by the $l^\infty$-distance $\mathrm{dist}_{l^\infty}=d\times_{l^\infty}\mathsf{D}$.

Clearly, the map $h:X\to X\times S$ is continuous and $\mathbb{R}$-equivariant, namely: $h\circ\phi_t=(\phi_t\times\sigma_t)\circ h$, for all $t\in\mathbb{R}$, and in consequence is a factor map from $(X,(\phi_t)_{t\in\mathbb{R}})$ to $(H,((\phi\times\sigma)_t)_{t\in\mathbb{R}})$, where $H$ is to denote the image $h(X)$ which forms a closed and $\phi\times\sigma$-invariant subset of $X\times S$. Therefore one has $$\mathrm{h}_\mathsf{top}(H,((\phi\times\sigma)_t)_{t\in\mathbb{R}})\le\mathrm{h}_\mathsf{top}(X,(\phi_t)_{t\in\mathbb{R}}).$$

Intuitively, in the following estimate for our purpose we are actually limiting our attention to the inequality: $$\log\mathrm{part}_\epsilon(Y,\rho^\psi_{[0,mn]})\le\log\mathrm{part}_{\frac{\eta}{N}}(X,d^\phi_{[0,n]})+\log\mathrm{part}_{\frac{\eta}{N}}(S,\mathsf{D}^\sigma_{[0,n]})+\log\#A_{\eta,\lfloor\frac{n}{N}\rfloor,N}.$$ The counting lemma is then employed to deal with the term $\log\#A_{\eta,\lfloor n/N\rfloor,N}$. However, to avoid an unnecessary factor $2\cdot\mathrm{h}_\mathsf{top}(X,(\phi_t)_{t\in\mathbb{R}})$ appearing, one has to replace the first and second terms in the right-hand side by the above-mentioned $\log\mathrm{part}_{\eta/N}(H,(d\times_{l^\infty}\mathsf{D})^{\phi\times\sigma}_{[0,n]})$ with a single map used in the proof of the lemma below. Essentially, this handling is for notational simplicity and explicit demonstration, and of course, one may also use the continuity of the map $\theta(1,\cdot):X\to X$ along with $(X,d^\phi_{[0,n]})$, instead of $(S,\mathsf{D}^\sigma_{[0,n]})$, to formulate the calculation.

Note that Lemma \ref{mM} is being used implicitly in the proof.

\begin{lemma}
$$\log\mathrm{part}_\epsilon(Y,\rho^\psi_{[0,mNn]})\le\log\mathrm{part}_{\delta/N}(H,(d\times_{l^\infty}\mathsf{D})^{\phi\times\sigma}_{[0,Nn]})+\log\#A_{\eta,n,N}.$$
\end{lemma}
\begin{proof}
For every integer $1\le i\le n$ and $x\in X$ put $$\Theta(Ni,x)=\eta\cdot\left\lfloor\frac{\theta(Ni,x)}{\eta}\right\rfloor.$$ Since $\theta(Ni,x)$ is increasing in $i$, one has $(\Theta(Ni,x))_{i=1}^n\in A_{\eta,n,N}$.

Take a finite set $P$ with $\#P=\mathrm{part}_{\delta/N}(H,(d\times_{l^\infty}\mathsf{D})^{\phi\times\sigma}_{[0,Nn]})$ and a map $f:H\to P$ such that if $x,x^\prime\in X$ satisfy $f((x,(\theta(1,\phi_t(x)))_{t\in\mathbb{R}}))=f((x^\prime,(\theta(1,\phi_t(x^\prime)))_{t\in\mathbb{R}}))$ then they satisfy both $d^\phi_{[0,Nn]}(x,x^\prime)\le\delta/N$ and $\mathsf{D}^\sigma_{[0,Nn]}((\theta(1,\phi_t(x)))_{t\in\mathbb{R}},(\theta(1,\phi_t(x^\prime)))_{t\in\mathbb{R}})\le\delta/N$.

Now let $g:Y\to P\times A_{\eta,n,N}$ be a map defined by $$y\mapsto(f(h(x)),(\Theta(Ni,x))_{i=1}^n),$$ where $x\in X$ is to denote $\pi^{-1}(y)$ for $y\in Y$, and $x^\prime\in X$ is planned to denote $\pi^{-1}(y^\prime)$ for $y^\prime\in Y$ in the sequel.

Suppose that $y,y^\prime\in Y$ satisfy $g(y)=g(y^\prime)$. By the definition of $\Theta(Ni,x)$ it is clear that $$|\theta(Ni,x)-\theta(Ni,x^\prime)|\le\eta,\quad\quad\quad\forall\,i\in[0,n]\cap\mathbb{Z}.$$ The choice of $f$ gives $$d^\phi_{[0,Nn]}(x,x^\prime)\le\frac{\delta}{N}\le\delta,$$$$\mathsf{D}^\sigma_{[0,Nn]}((\theta(1,\phi_t(x)))_{t\in\mathbb{R}},(\theta(1,\phi_t(x^\prime)))_{t\in\mathbb{R}})\le\frac{\delta}{N}\le\frac{\eta}{N}.$$ Note that $$\theta(i,x)=\sum_{l=0}^{i-1}\theta(1,\phi_l(x)),\quad\quad\quad\forall\;i\in\mathbb{N},$$ and in consequence for any $q\in[0,n-1]\cap\mathbb{Z}$ and $j\in[1,N]\cap\mathbb{Z}$ one has $$\theta(qN+j,x)=\theta(qN,x)+\sum_{l=qN}^{qN+j-1}\theta(1,\phi_l(x)),$$ which applies to $x^\prime$, too. It follows that $$|\theta(i,x)-\theta(i,x^\prime)|\le\eta+N\cdot\frac{\eta}{N}=2\eta,\quad\quad\quad\forall\;i\in[0,Nn]\cap\mathbb{Z}.$$ Moreover, note that $$\theta(t,x)=\theta(i,x)+\theta(r,\phi_i(x)),\quad$$ for any $t\in\mathbb{R}$ being represented as $t=i+r$, where $i=\lfloor t\rfloor$ and $0\le r<1$. By the choice of $\delta$, this implies that $$|\theta(t,x)-\theta(t,x^\prime)|\le2\eta+\eta=3\eta,\quad\quad\quad\forall\;t\in[0,Nn].$$ By the choice of $\eta$, $$\rho(\psi_{\theta(t,x)}(\pi(x^\prime)),\psi_{\theta(t,x^\prime)}(\pi(x^\prime)))\le\epsilon/2,\quad\quad\forall\;t\in[0,Nn],$$ and further, by the choice of $\delta$ again, $$\rho(\psi_{\theta(t,x)}(y),\psi_{\theta(t,x)}(y^\prime))\le\rho(\pi(\phi_t(x)),\pi(\phi_t(x^\prime)))+\rho(\psi_{\theta(t,x)}(\pi(x^\prime)),\psi_{\theta(t,x^\prime)}(\pi(x^\prime)))$$$$\;\quad\quad\quad\le\frac\epsilon2+\frac\epsilon2=\epsilon,\quad\quad\quad\quad\quad\quad\forall\;t\in[0,Nn].$$ By the definition of $m$$$\mathrm{part}_\epsilon(Y,\rho^\psi_{[0,mNn]})\le\#(P\times A_{\eta,n,N})=(\#P)\cdot(\#A_{\eta,n,N})$$ and hence the statement follows.
\end{proof}

Applying this lemma to every $n,N\in\mathbb{N}$ one deduces that for each $N\in\mathbb{N}$$$\lim_{n\to\infty}\frac{\log\mathrm{part}_\epsilon(Y,\rho^\psi_{[0,mn]})}{n}\le\lim_{n\to\infty}\frac{\log\mathrm{part}_{\delta/N}(H,(d\times_{l^\infty}\mathsf{D})^{\phi\times\sigma}_{[0,n]})}{n}+\frac1N\cdot\lim_{n\to\infty}\frac{\log\#A_{\eta,n,N}}{n}.$$ Note that all the limits involved in the above inequality exist (in fact, the first and second limits exist for $n$ ranging over $\mathbb{R}$), and hence one may replace a (sub)sequence (along which a limit was taken) contained in $\mathbb{N}$ (or $\mathbb{R}$) by any of the others (here for the first and second limits, more specifically, $\{Nn\}_{n=1}^{+\infty}$ has been replaced with $\{n\}_{n=1}^{+\infty}$). By Lemma \ref{counting} one concludes (by letting $N\to\infty$) that $$m\cdot\lim_{n\to\infty}\frac{\log\mathrm{part}_\epsilon(Y,\rho^\psi_{[0,n]})}{n}\le\mathrm{h}_\mathsf{top}(H,((\phi\times\sigma)_t)_{t\in\mathbb{R}}).$$ Since $\epsilon>0$ was taken arbitrarily, $$m\cdot\mathrm{h}_\mathsf{top}(Y,(\psi_s)_{s\in\mathbb{R}})\le\mathrm{h}_\mathsf{top}(H,((\phi\times\sigma)_t)_{t\in\mathbb{R}})\le\mathrm{h}_\mathsf{top}(X,(\phi_t)_{t\in\mathbb{R}}),$$ as desired.

\subsection{Estimate for $M$}
Here to show the inequality involving $M$ we adopt a brief but more comprehensive treatment for topological entropy, which still applies properly to the proof of the statement for mean dimension (but not vice versa). Consider the inverse (with respect to $\theta$) time-reparameterization $\tau$ (induced by $\pi^{-1}:Y\to X$) between $(Y,(\psi_s)_{s\in\mathbb{R}})$ and $(X,(\phi_t)_{t\in\mathbb{R}})$, i.e., a uniquely determined map $\tau:\mathbb{R}\times Y\to\mathbb{R}$ possessing all the properties (i), (ii), (iii), (iv) similar to those for $\theta:\mathbb{R}\times X\to\mathbb{R}$ as listed in Section 2. Note that an easily-deduced but important connection between $\tau$ and $\theta$ is the following equalities: $$s=\theta(\tau(s,\pi(x)),x),\quad t=\tau(\theta(t,x),\pi(x)),\quad\quad\forall\;s,t\in\mathbb{R},\quad\forall\;x\in X.$$ In particular, $$1=\tau(\theta(1,\pi^{-1}(y)),y)\le\tau(M,y),\quad\quad\quad\forall\;y\in Y.$$ Applying the estimate for $m$ to the above-mentioned lower bound $1$ for $\tau(M,\cdot)$ (in fact, one may see further $\min_{y\in Y}\tau(M,y)=1$) one concludes that $$\mathrm{h}_\mathsf{top}(X,(\phi_t)_{t\in\mathbb{R}})\le\mathrm{h}_\mathsf{top}(Y,(\psi_s^M)_{s\in\mathbb{R}})=M\cdot\mathrm{h}_\mathsf{top}(Y,(\psi_s)_{s\in\mathbb{R}}),$$ as desired.

\subsection*{Remark}
The idea contained in this section along with the corresponding ingredient of a modified approach to the definition of topological entropy (for all the actions of amenable groups) in Section 3 has some novelty, and (probably) leads to some potential applications elsewhere. The authors would hope that this alternative could grow to be a standard method (which, at least, is being hoped to provide another choice different from traditional ones) for addressing issues of (in particular, for estimating) topological entropy.

\section{Weakly equivalent flows with fixed points:\\a stronger example}
In this section, both topological entropy and mean dimension are being involved. The main theorem of this section is an explicitly-constructed example with a purely topological proof of its properties. Note that this result strengthens the one given by Ohno with a view towards mean dimension, because a flow of positive mean dimension can only have infinite topological entropy \cite{LW00}.

\begin{theorem}
There exists a pair of weakly equivalent flows such that one of them has infinite mean dimension and the other has zero topological entropy.
\end{theorem}

Generally speaking, our strategy is perturbation, and our technique refines the procedure suggested in Ohno's construction. Specifically, in Ohno's example the alphabet consists of two distinct points (i.e., its base used in the construction is a subshift over two symbols), while in our construction we will replace one of them with a closed interval bounded away from the other which we leave as is. Intuitively, Ohno's example comes essentially to be a factor of ours (indeed, one may induce a factor map by identifying all the points belonging to the compact interval with a single point different from the other point).

In what follows, we refine Ohno's construction, and in particular, we adopt an elementary method in the proof of the required properties, which proves to be much more straightforward than, and which differs substantially from, the measure-theoretic approach provided by Ohno, even if one's attention is only focused on the level of complexity characterized by topological entropy.

For notational simplicity, we turn to proving a seemingly weaker but actually equivalent statement:
\begin{itemize}\item
There exist two weakly equivalent flows $(X,(\phi_t)_{t\in\mathbb{R}})$ and $(Y,(\psi_s)_{s\in\mathbb{R}})$ such that $(X,(\phi_t)_{t\in\mathbb{R}})$ has zero topological entropy while $(Y,(\psi_s)_{s\in\mathbb{R}})$ has positive mean dimension.
\end{itemize}
In fact, it is very easy to see that this statement implies the above theorem. The reason is as follows: First of all, one picks a pair of weakly equivalent flows $(X,(\phi_t)_{t\in\mathbb{R}})$ and $(Y,(\psi_s)_{s\in\mathbb{R}})$ such that $(X,(\phi_t)_{t\in\mathbb{R}})$ has zero topological entropy and $(Y,(\psi_s)_{s\in\mathbb{R}})$ has positive mean dimension. Then, one can take a disjoint union of countably many flows $(X,(\phi_{Nt})_{t\in\mathbb{R}})$, for $N$ ranging over $\mathbb{N}$, with diameters tending to $0$ (as $N\to\infty$), which converge to (by adding) one fixed point. And similarly, one takes a disjoint union of countably many flows $(Y,(\psi_{Ns})_{s\in\mathbb{R}})$, for $N$ running over $\mathbb{N}$, with diameters tending to $0$ (as $N\to\infty$), such that they converge eventually to a fixed point added. Now one may see directly that these two (unions of) flows are still weakly equivalent, and meanwhile, the former has zero topological entropy while the latter has infinite mean dimension. Therefore any example constructed for the reduced statement will lead finally to a one valid for the theorem.

\subsection{Construction}
Roughly speaking, the construction will be fulfilled with two suspensions (but with an added twist) generated by two distinct roof functions, respectively, over a base (namely, given by a homeomorphism defined on a compact metrizable space) in common.

\subsubsection*{The base $(B,\sigma)$}
Put $I=[0,1]\cup\{-1\}$ and equip the set $I$ with the usual Euclidean distance, which will be used as the alphabet, where the compact interval $[0,1]$, along with a positive density of its appearance to the coordinate projections, is to guarantee the dimension requirement, and where the point $-1$ is planned to be used when generating a fixed point (actually being the unique singularity) near which one can make the growth rate of a roof function be lifted sufficiently high, so as to lower topological entropy for a flow.

Let $\mathcal{I}=I^\mathbb{Z}$ be the (countable) self-product of $I$, endowed with the product topology (which consequently becomes a compact metrizable space). Denote by $\sigma:\mathcal{I}\to\mathcal{I}$ the homeomorphism of $\mathcal{I}$ given by the shift action on the product space $I^\mathbb{Z}$. This is typically referred to as the full shift (of the group $\mathbb{Z}$) over the alphabet $I$. The base $(B,\sigma)$ will be chosen among subshifts of $(\mathcal{I},\sigma)$, i.e., $B$ will be constructed as a closed and $\sigma$-invariant subset of $\mathcal{I}$. More specifically, $B$ is defined to be $\overline{\sigma(E)}$ for some $E\subset\mathcal{I}$, the orbit closure (under the action of $\sigma$) of a subset $E$ of $\mathcal{I}$.

Next it suffices to define the set $E\subset\mathcal{I}$, which will be of the form $E=\prod_{i=-\infty}^{+\infty}F_i$ which thus is referred to as a (two-sided) \textit{string}, where $F_i\subset I$ for every $i\in\mathbb{Z}$. Let $F_0=\{-1\}$ and $F_{-i}=F_i$ for each $i\in\mathbb{N}$. It remains to describe $F_i$ for all $i\in\mathbb{N}$, which will be characterized by induction. Since the inductive procedure below has encompassed the whole range of positive integers (as coordinates), any $F_i\subset I$ (needed in the string $\prod_{i=1}^{+\infty}F_i$) in consequence is indeed uniquely determined, and moreover, is equal to either $\{-1\}$ or $[0,1]$ which are referred to as two (set-valued) \textit{letters} to constitute a \textit{word} (namely a finite substring of $E$). Further, it follows immediately from the constructive process that the string $E$ satisfies the following conditions:
\begin{itemize}
\item
for any $j\in\mathbb{Z}$ and any $n\in\mathbb{N}$ the word $\prod_{i=j}^{j+4\cdot3^n}F_i$ contains at least $(2n-1)$ consecutive $\{-1\}$ (i.e., $\{-1\}^{2n-1}$) as a subword;
\item
for any $n\in\mathbb{N}$ there exists a word of length $2\cdot3^{n-1}$, in which the number of $[0,1]$ is $p_n=(3^{n-1}+1)/2$.
\end{itemize}
Both of these two desired properties have nothing to do with the weak equivalence relation. Exactly, the former is for the flow $(X,(\phi_t)_{t\in\mathbb{R}})$ being of topological entropy zero, while the latter is to ensure that the flow $(Y,(\psi_s)_{s\in\mathbb{R}})$ has positive mean dimension.

Put $H_1=\{-1\}\times[0,1]$ and let $H_2=H_1\times\tilde{H_1}\times H_1$, where $\tilde{H_1}$ is defined to be $\{-1\}\times\{-1\}$. Inductively, if $H_n$ is supposed to be defined, for some $n\in\mathbb{N}$, and is assumed to contain at least $(2n-1)$ consecutive $\{-1\}$ as a subword, then $H_{n+1}$ is defined as $H_{n+1}=H_n\times\tilde{H_n}\times H_n$, where $\tilde{H_n}$ is defined by replacing some $[0,1]$ in $H_n$ with $\{-1\}$ such that $H_{n+1}$ contains at least $(2n+1)$ consecutive $\{-1\}$ as a subword (which may be chosen as a one located in the midst of a subword $\{-1\}$ and a subword of the maximum length consisting of consecutive $\{-1\}$, in the word $H_{n+1}$, and hence contributes two more $\{-1\}$ to $H_{n+1}$).

Note that for each $n\in\mathbb{N}$ the length of the word $H_n$ is equal to $2\cdot3^{n-1}$, and that the number of $[0,1]$ in $H_n$ is $p_n$. Clearly, the string $\prod_{i=1}^{+\infty}F_i$ is eventually (well) defined by (re-enumerating the limit of) $\{H_n\}_{n\in\mathbb{N}}$ (as $n\to\infty$). More precisely, $H_n$ defines the word $\prod_{i=1}^{2\cdot3^{n-1}}F_i$, for every $n\in\mathbb{N}$.

Compared to \cite{O80}, here we adopt almost the same constants, for the reader's convenience, as what Ohno used in the construction. But we need point out that in contrast to it, although for any $n\in\mathbb{N}$, $H_n$ appears infinitely many times as a word within the string $\prod_{i=1}^{+\infty}F_i$, neither this fact nor a further modification (mentioned in Ohno's construction \cite{O80}) to the above-defined string is needed.

\subsubsection*{The flows and weak equivalence}
Denote by $\{-1\}^\mathbb{Z}$ the (only) fixed point of $\sigma$ in $B$. Put $A_\star=B\setminus\{\{-1\}^\mathbb{Z}\}$ which becomes a locally compact metrizable space. Let $\gamma:A_\star\to(0,+\infty)$ be a continuous function (\textit{not} necessarily bounded). Consider the quotient space $A_\star^\gamma$ of $\{(u,x):0\le u\le\gamma(x),x\in A_\star\}$, where the equivalence relation $\sim$ is defined by $(\gamma(x),x)\sim(0,\sigma(x))$. Now a flow $(\xi^\gamma_r)_{r\in\mathbb{R}}$ defined on the locally compact metrizable space $A_\star^\gamma$ is naturally induced as follows: $\xi^\gamma_r(u,x)=(u+r,x)$, for $-u\le r<\gamma(x)-u$, where any representative $(u,x)$ should be understood as $(u,x)/\sim$ in $A_\star^\gamma$ provided the notation for the action of $\xi^\gamma_r$ extends to all $r\in\mathbb{R}$.

Note that this convention will still be valid for the formulation in the sequel (i.e., a point in the quotient space $A_\star^\gamma$ is written only as any of its representatives, when removing the symbol $/\sim$ of the equivalence relation from the notation for the points in the quotient space).

Let $A^\gamma=A_\star^\gamma\cup\{\ast\}$ be a one-point compactification of the locally compact metrizable space $A_\star^\gamma$. Clearly, the flow $(\xi^\gamma_r)_{r\in\mathbb{R}}$ extends (continuously) to the compact metrizable space $A^\gamma$ if $\ast$ is further required to be a fixed point of $(\xi^\gamma_r)_{r\in\mathbb{R}}$ in $A^\gamma$: $\xi^\gamma_r(\ast)=\ast$, for all $r\in\mathbb{R}$.

The weak equivalence relation between $(X,(\phi_t)_{t\in\mathbb{R}})$ and $(Y,(\psi_s)_{s\in\mathbb{R}})$ (being both constructed soon) is obvious, and is generally true for all the flows chosen as above. In fact, from the construction one may see that for any continuous functions $\gamma,\gamma^\prime:A_\star\to(0,+\infty)$ the flows $(\xi^\gamma_r)_{r\in\mathbb{R}}$ and $(\xi^{\gamma^\prime}_r)_{r\in\mathbb{R}}$, defined on $A^\gamma$ and $A^{\gamma^\prime}$ respectively, are weakly equivalent to each other, as a weak equivalence relation between $(A^\gamma,(\xi^\gamma_r)_{r\in\mathbb{R}})$ and $(A^{\gamma^\prime},(\xi^{\gamma^\prime}_r)_{r\in\mathbb{R}})$ is given by $\ast\mapsto\ast$ along with $(u,x)\mapsto(u\cdot\frac{\gamma^\prime(x)}{\gamma(x)},x)$, for $x\in A_\star$ and $0\le u<\gamma(x)$.

\subsection{The first flow $(Y,(\psi_s)_{s\in\mathbb{R}})$ and its mean dimension}
Now the first roof is simply taken to be the constant function $\gamma_1\equiv1$ on $A_\star$, and the flow $(Y,(\psi_s)_{s\in\mathbb{R}})$ is then defined by $(A^{\gamma_1},(\xi^{\gamma_1}_r)_{r\in\mathbb{R}})$.

Note that the base $(B,\sigma)$ can be considered as a subsystem of $(Y,\psi_1)$; more precisely, under the homeomorphism $\psi_1:Y\to Y$ (given by the time-$1$ map of the flow $(Y,(\psi_s)_{s\in\mathbb{R}})$), the set $B$ can be regarded as a closed and $\psi_1$-invariant subset of $Y$, and meanwhile, $\psi_1:B\to B$ (the restriction of $\psi_1$ to $B$) again gives $\sigma:B\to B$. Since $\mathrm{mdim}(Y,(\psi_s)_{s\in\mathbb{R}})=\mathrm{mdim}(Y,\psi_1)$ (for details see Subsection 3.5 in company with Subsection 3.4), it suffices to show that $\mathrm{mdim}(B,\sigma)$ is positive.

Since the string $E$ is of the type $\prod_{i=-\infty}^{+\infty}F_i$ for $F_i$ ($i\in\mathbb{Z}$) being either $\{-1\}$ or $[0,1]$ (which easily induces some \textit{continuous} distance-increasing map being helpful in the sense below), by the second property of $E$$$\mathrm{Widim}_\epsilon([0,1]^{p_n},\mathsf{dist}_{l^\infty})\le\mathrm{Widim}_\epsilon(B,\mathsf{D}^\sigma_{[0,2\cdot3^{n-1}-1]\cap\mathbb{Z}}),\quad\quad\forall\;n\in\mathbb{N},\quad\forall\;\epsilon>0,$$ where $\mathsf{dist}$ is the usual Euclidean distance on $I$ and where $\mathsf{D}$ is a typically-used distance on $\mathcal{I}$ as introduced before (which was applied to countable products there for simplicity, see Subsection 3.3). By the Lebesgue lemma (which is useful here for all $0<\epsilon<1$, see Subsection 3.5) one has $\mathrm{mdim}(B,\sigma)\ge\frac14$, as required.

\subsection{The second flow $(X,(\phi_t)_{t\in\mathbb{R}})$ and its topological entropy}
The second roof function $\gamma_0:A_\star\to(0,+\infty)$ is picked as follows: for every $n\in\mathbb{N}$, if $x\in Q_n\setminus Q_{n+1}$ then let $\gamma_0(x)=n\cdot4\cdot3^n$; while for all $x\in A_\star\setminus Q_1$ let $\gamma_0(x)=1$, where $Q_n$ is to denote the set $Q_n=\{q=(q_j)_{j\in\mathbb{Z}}\in B:\,q_i=-1,\;\forall i\in[-n+1,n-1]\cap\mathbb{Z}\}$, for each $n\in\mathbb{N}$. It is clear that the function $\gamma_0:A_\star\to(0,+\infty)$ is continuous. The flow $(X,(\phi_t)_{t\in\mathbb{R}})$ is now defined by $(A^{\gamma_0},(\xi^{\gamma_0}_r)_{r\in\mathbb{R}})$.

Next we need prove that $\mathrm{h}_\mathsf{top}(X,\phi_1)=0$. Unlike the proof given in Section 5, note that in relation to the topological entropy of this concrete flow, no matter which definition for topological entropy the proof below will use, there is no essential difference. So here we choose, on the contrary, to follow Bowen's definition for topological entropy, so as to show the reader in what kind of situations the key difference can be made.

First note that as $n\to\infty$, $(t_n,x_n)\to\ast$ in $X=A^{\gamma_0}$ if and only if $x_n\to\{-1\}^\mathbb{Z}$ in $B$.

Take an arbitrary $0<\epsilon<1$ (note that here ``$<1$'' is exactly to distinguish between $-1$ and a point in $[0,1]$) and fix it temporarily. Find some $N\in\mathbb{N}$ (sufficiently large) such that $\mathrm{dist}(\ast,(t,x))<\epsilon$ for all $t\ge N\cdot4\cdot3^N-1$ and $x\in A_\star$, where $\mathrm{dist}$ is a distance on $X$, which gives the topology of $X$ and which is bounded by (the restriction of) the $l^\infty$-distance (to $A^{\gamma_0}$). Find some $L\in\mathbb{N}$ (sufficiently large) such that if two points $q=(q_j)_{j\in\mathbb{Z}},q^\prime=(q^\prime_j)_{j\in\mathbb{Z}}\in B$ satisfy that $|q_i-q^\prime_i|\le\epsilon$ for all $i\in[-L,L]\cap\mathbb{Z}$, then they satisfy $\mathsf{D}(q,q^\prime)\le2\epsilon$.

For any integer $n>N$, a $(2\epsilon)$-spanning set for the space $X$ with respect to the distance $\mathrm{dist}^\phi_{[0,n\cdot4\cdot3^n-1]\cap\mathbb{Z}}$ can be explicitly constructed as the union of the following three subsets of $X$: $\{\ast\}$, $G\cap X$ and $V\cap X$, where $$G=\left(\bigcup_{j=0}^{\left\lfloor\frac1\epsilon\right\rfloor}\bigcup_{i=0}^{n\cdot4\cdot3^n}\left\{i+j\cdot\epsilon\right\}\right)\times\left(\bigcup_{i_{-4\cdot3^{n+1}-L-1},\dots,i_{4\cdot3^{n+1}+L+1}\in I^\epsilon_{-1}}\left\{q^{i_{-4\cdot3^{n+1}-L-1},\dots,i_{4\cdot3^{n+1}+L+1}}\right\}\right),$$ in which $I^\epsilon_{-1}\subset I$ is defined by $$I^\epsilon_{-1}=\left\{-1\right\}\cup\left\{j\cdot\epsilon:j\in\left[0,\left\lfloor\frac1\epsilon\right\rfloor\right]\cap\mathbb{Z}\right\}$$ and in which every point $q^{i_{-4\cdot3^{n+1}-L-1},\dots,i_{4\cdot3^{n+1}+L+1}}$ is defined as follows: if the set $$B\cap\left(\bigcap_{l=-4\cdot3^{n+1}-L-1}^{4\cdot3^{n+1}+L+1}\left\{q=(q_j)_{j\in\mathbb{Z}}\in\mathcal{I}:q_l=i_l\right\}\right)$$ is nonempty, then take a point $q^{i_{-4\cdot3^{n+1}-L-1},\dots,i_{4\cdot3^{n+1}+L+1}}$ arbitrarily from this set, otherwise take an arbitrary $q^{i_{-4\cdot3^{n+1}-L-1},\dots,i_{4\cdot3^{n+1}+L+1}}$ in the set $$\bigcap_{l=-4\cdot3^{n+1}-L-1}^{4\cdot3^{n+1}+L+1}\left\{q=(q_j)_{j\in\mathbb{Z}}\in\mathcal{I}:q_l=i_l\right\}$$ (provided the intersection $B\cap(\cap_{l=-4\cdot3^{n+1}-L-1}^{4\cdot3^{n+1}+L+1}\{q=(q_j)_{j\in\mathbb{Z}}\in\mathcal{I}:q_l=i_l\})$ is empty); and where $$V=\bigcup_{j=0}^{\left\lfloor\frac1\epsilon\right\rfloor}\bigcup_{i=0}^{n\cdot4\cdot3^n}\left\{h^{i,j}=(u^i_j,v):u^i_j=\left((n+1)\cdot4\cdot3^{n+1}-i\right)+j\cdot\epsilon,\;v=(v_l)_{l\in\mathbb{Z}}\in B\right\},$$ in which $v=(v_l)_{l\in\mathbb{Z}}$ is picked as follows: if $l\in[-n,n]\cap\mathbb{Z}$, then define $v_l=-1$; if $l\in\{-n-1,n+1\}$, then take some $v_l\in[0,1]$, while for all $l\in\mathbb{Z}\setminus[-n-1,n+1]$, take some $v_l\in I$, but meanwhile, the defined $v=(v_l)_{l\in\mathbb{Z}}$ should finally be such that $v=(v_l)_{l\in\mathbb{Z}}\in B$ (such a choice clearly exists).

Note that there are many ``ghosts of non-existing members'' being put to the expressions for $G$ and $V$, namely some elements which are in $G$ or $V$ as described above but which do not belong to $X$. As a result, we invited ``$\cap X$'' to be with both $G$ and $V$ (i.e., to the second and third participants in the union), so as to emphasize this fact which clearly does not have any influence on that formulation at all.

Note that $$\#G\le\left(\left\lfloor\frac1\epsilon\right\rfloor+1\right)\cdot\left(n\cdot4\cdot3^n+1\right)\cdot\left(\left\lfloor\frac1\epsilon\right\rfloor+2\right)^{2\cdot4\cdot3^{n+1}+2L+3},$$$$\#V\le\left(\left\lfloor\frac1\epsilon\right\rfloor+1\right)\cdot\left(n\cdot4\cdot3^n+1\right).$$ As $0<\epsilon<1$ was taken arbitrarily and has been fixed for the moment, a straightforward calculation by definition (and by counting the sum of the cardinalities of these participants in the union $\{\ast\}\cup G\cup V$, which bounds the sum of those in $\{\ast\}\cup(G\cap X)\cup(V\cap X)$) suffices to show that $\mathrm{h}_\mathsf{top}(X,\phi_1)=0$: $$\frac{\log\mathrm{span}_{2\epsilon}(X,\mathrm{dist}^\phi_{[0,n\cdot4\cdot3^n-1]\cap\mathbb{Z}})}{n\cdot4\cdot3^n}\le\frac{\log(1+\#G+\#V)}{n\cdot4\cdot3^n}\to0\quad\quad\quad(n\to+\infty),$$ as desired.

Thus now it remains to be seen that $\{\ast\}\cup(G\cap X)\cup(V\cap X)$ is indeed a $(2\epsilon)$-spanning set for $X$ with respect to the distance $\mathrm{dist}^\phi_{[0,n\cdot4\cdot3^n-1]\cap\mathbb{Z}}$. In fact, this follows directly from the definition of $B=\overline{\sigma(E)}$, the first property of $E=\prod_{i=-\infty}^{+\infty}F_i$, and the construction of the sets $G$ and $V$. But in what follows, we explain the reason in detail. This will complete the proof.

Let us take any point $(t,x)\in X$, which is now simply referred to as a \textit{traveller} (the description will benefit our explanation). When traveling during the (discrete) time interval $[0,n\cdot4\cdot3^n-1]\cap\mathbb{Z}$ (starting from some place in $X$), there are three types of activities for a traveller: walking, climbing and jumping. The base $B$ is considered as the \textit{ground}, while the only point $\ast$ of the first participant in the union is considered as the \textit{sun} which never goes down to the ground.

More precisely, since $\gamma_0(q)=1$ for all those $q=(q_j)_{j\in\mathbb{Z}}\in B$ with $q_0\ne-1$, any traveller, at any moment, either walks at a steady pace on the ground or on a level parallel to the ground (i.e., one step per second, to the next place at the same level), or climbs up (namely along the direction vertical to the ground), or jumps down (to the next place on the ground or some level, parallel to the ground, at a height of at most $1$). Note that here the \textit{second} (namely the time-$1$ map $\phi_1:X\to X$) is being used as the unit of time, i.e., if a traveller stays in a place $(\underline{t},\underline{x})\in X$, then after $1$ second, the one will move to the \textit{next} place $\phi_1((\underline{t},\underline{x}))$.

Further, by the first property of $E=\prod_{i=-\infty}^{+\infty}F_i$ and by the definition of $B=\overline{\sigma(E)}$, a traveller \textit{has} to climb up (for a sub-duration) during the time interval of traveling, and moreover, as indicated exactly in the length of the time interval, never jumps down twice from a very high place (beyond the top, near the sun), i.e., the traveller $(t,x)\in X$ has at most one opportunity of jumping down from a place at a height beyond the very \textit{top} $(n\cdot4\cdot3^n,x)\in X$. Here recall that the integer $n>N$ has been arbitrarily fixed for the moment, and also note that all the points in the space $X$ are usually used as positions (namely, being referred to as some \textit{places}), but still sometimes used as individuals (e.g., the traveller, and the ones from $(G\cup V)\cap X$).

Generally speaking, one coming from $G\cap X$ or $V\cap X$ is to be together with the traveller when taking part in any of the activities of the three kinds. To be precise, the set $G$ may be regarded as a group of \textit{travelling companions}, namely being responsible for moving (including walking, climbing, and jumping) \textit{normally} with the traveller (here the normal situation means that the traveller does not need to start from a place higher than the very top $(n\cdot4\cdot3^n,x)\in X$), among which the traveller should choose according to the starting place of traveling, while the set $V$ may be regarded as a group of \textit{experts} being exactly responsible for jumping down, at any possible moment (namely at the $l$-th second, for every $l$ in the time interval of traveling), in company with the traveller who starts from a place higher than $n\cdot4\cdot3^n$ ($\epsilon$-close to the sun), i.e., the traveller \textit{must} be accompanied by an expert from $V$ when jumping down, provided the one starts from a place at a height beyond the very top $(n\cdot4\cdot3^n,x)\in X$.

Clearly, there are now three cases listed as follows. In each of the cases a proper candidate is directly elected to be in company with the traveller (i.e., being $2\epsilon$-close to the traveller in $X$ during the whole time interval of traveling $[0,n\cdot4\cdot3^n-1]\cap\mathbb{Z}$), as precisely described above.
\begin{itemize}
\item
Case 1. The $\{\ast\}$ case: the sun!
\begin{itemize}
\item
Assume $(t,x)=\ast$. This is the trivial case (one may take $\ast$ itself as required).
\end{itemize}
\item
Case 2. The $G$ case: the traveller starts from a place at a height of at most $n\cdot4\cdot3^n$.
\begin{itemize}
\item
If $(t,x)\in[0,n\cdot4\cdot3^n]\times A_\star$, then the definition of $G$ simply gives a proper choice of travelling companions being with the traveller $(t,x)$ through the time interval of traveling, i.e., being $2\epsilon$-close to $(t,x)$ at any moment $l$ within $[0,n\cdot4\cdot3^n-1]\cap\mathbb{Z}$. Note that besides the definition of the set $G$, the definition of the continuous function $\gamma_0:A_\star\to(0,+\infty)$ along with the involved sets $Q_m$, for $m\in\mathbb{N}$, has been implicitly used, too.
\end{itemize}
\item
Case 3. The $V$ case: the traveller starts from a place higher than $n\cdot4\cdot3^n$.
\begin{itemize}
\item
Suppose that $(t,x)\in(n\cdot4\cdot3^n,+\infty)\times A_\star$. In relation to this case, there exist two sub-cases as listed below:
\begin{itemize}
\item[(i)]
If the traveller $(t,x)$ does not jump down for the duration of the traveling (namely if $\phi_l((t,x))\in(n\cdot4\cdot3^n,+\infty)\times A_\star$, for all $l\in[0,n\cdot4\cdot3^n-1]\cap\mathbb{Z}$), then one can still simply take the sun $\ast$ as desired.
\item[(ii)]
Otherwise, someone else outside $G$, namely a proper expert from $V$, has to be invited; and note that this is the exact reason why the set $V$ was introduced there. Now by the definition of $V$, the jumping moment (along with the initial position) for the traveller $(t,x)$ leads exactly to some expert $h^{i,j}$ from $V$ as desired, i.e., there exist some $i\in[0,n\cdot4\cdot3^n-1]\cap\mathbb{Z}$ (being referred to as the jumping moment) and $j\in[0,\lfloor1/\epsilon\rfloor]\cap\mathbb{Z}$ (being used to choose the nearest expert $h^{i,j}$, where $i\in[0,n\cdot4\cdot3^n-1]\cap\mathbb{Z}$ has been found, from $V$) such that $\phi_l((t,x))\in(n\cdot4\cdot3^n,+\infty)\times A_\star$ for every $l\in[0,i-1]\cap\mathbb{Z}$, while $\phi_l((t,x))\in[0,n\cdot4\cdot3^n]\times A_\star$, for every $l\in[i,n\cdot4\cdot3^n-1]\cap\mathbb{Z}$, and meanwhile, such that the expert $h^{i,j}\in V$ is $2\epsilon$-close to the traveller $(t,x)\in X$ at any moment $l$ through the whole traveling-time interval $[0,n\cdot4\cdot3^n-1]\cap\mathbb{Z}$ (which in consequence indicates a proper expert $h^{i,j}\in V$). Here one should note that before the traveller $(t,x)$ and the expert $h^{i,j}$ jump down, both of them are $\epsilon$-close to the sun; but after jumping down together to their next places, respectively, they are still $2\epsilon$-close to each other at any moment until the travelling ends (because of the definition of the expert $h^{i,j}\in V$, besides which, the definition of $\gamma_0:A_\star\to(0,+\infty)$ along with $Q_m$, for $m\in\mathbb{N}$, is also used implicitly here).
\end{itemize}
\end{itemize}
\end{itemize}

\noindent\textbf{Remark.}
As mentioned before, for topological entropy Sun--Young--Zhou \cite{SYZ09} found such an example (in relation to the reduced statement) in the class of smooth flows, and thus it seems reasonable if one further intends to require our above-constructed example to be smooth. But this immediately turns out to be not proper, as a differentiable flow on a compact manifold cannot have infinite topological entropy, and hence, has to be of mean dimension zero. Nonetheless, note that it is also possible for one to find such a pair of weakly equivalent flows, one of which has infinite mean dimension (and hence, infinite metric mean dimension) while the other has zero topological entropy dimension (and hence, zero topological entropy). The notion of \textit{topological entropy dimension} was introduced by Dou--Huang--Park in \cite{DHP11} around 2011, in company with a very nice and systematic investigation.

\medskip


\begin{thebibliography}{9999999}

\bibitem[DHP11]{DHP11}
Dou Dou, Wen Huang, Kyewon Koh Park.
Entropy dimension of topological dynamical systems.
Transactions of the American Mathematical Society 363 (2011), 659--680.

\bibitem[LW00]{LW00}
Elon Lindenstrauss, Benjamin Weiss.
Mean topological dimension.
Israel Journal of Mathematics 115 (2000), 1--24.

\bibitem[O80]{O80}
Taijiro Ohno.
A weak equivalence and topological entropy.
Publications of the Research Institute for Mathematical Sciences (Kyoto University) 16 (1980), 289--298.

\bibitem[SYZ09]{SYZ09}
Wenxiang Sun, Todd Young, Yunhua Zhou.
Topological entropies of equivalent smooth flows.
Transactions of the American Mathematical Society 361 (2009), 3071--3082.

\bibitem[W82]{W82}
Peter Walters.
An introduction to ergodic theory.
Springer-Verlag (1982).

\end{thebibliography}
\end{document}